\documentclass[12pt,leqno]{amsart}
\usepackage{amsmath, amssymb, amscd, amsfonts, xypic,   stmaryrd, turnstile, mathrsfs}

\usepackage{hyperref}

\newtheorem{theorem}{Theorem}[section]
\newtheorem{lemma}[theorem]{Lemma}
\newtheorem{proposition}[theorem]{Proposition}
\newtheorem{corollary}[theorem]{Corollary}

\theoremstyle{definition}
\newtheorem{definition}[theorem]{Definition}
\newtheorem{example}[theorem]{Example}
\newtheorem{remark}[theorem]{Remark}

\def\car#1,#2,#3,#4{
$$
   \CD
   #1           @>{}>>          #2        \\
   @V{{}}VV                  @VV{{}}V  \\
   #3        @>{{}}>>   #4
   \endCD
   $$}

\def\Car#1,#2,#3,#4,#5,#6,#7,#8{
$$
   \CD
   #1           @>#2>>          #3        \\
   @V{#4}VV                  @VV{#5}V  \\
   #6        @>{#7}>>   #8
   \endCD
   $$}

\begin{document}

\title[Quasi-Pr\"ufer extensions]{Quasi-Pr\"ufer   extensions of rings}

\author[G. Picavet and M. Picavet]{Gabriel Picavet and Martine Picavet-L'Hermitte}
\address{Universit\'e Blaise Pascal \\
Laboratoire de Math\'ematiques\\ UMR6620 CNRS  \\ 24, avenue des Landais\\
BP 80026 \\ 63177 Aubi\`ere CEDEX \\ France}

\email{Gabriel.Picavet@math.univ-bpclermont.fr}
\email{picavet.gm(at)wanadoo.fr}

\begin{abstract}  We introduce quasi-Pr\"ufer ring extensions, in order to relativize quasi-Pr\"ufer domains and to take also into account some contexts in recent papers, where such extensions appear in a hidden form. An extension is quasi-Pr\"ufer if and only if it is an INC pair.  The class of these extensions has  nice stability properties. We also define almost-Pr\"ufer extensions that are quasi-Pr\"ufer, the converse being not true. Quasi-Pr\"ufer extensions are closely linked to finiteness properties of fibers. Applications are given for FMC extensions, because they are quasi-Pr\"ufer.

\end{abstract} 
\thanks{}


\keywords  {flat epimorphism, FIP, FCP extension, minimal extension, integral extension, Morita, Pr\"ufer hull, support of a module}

\maketitle

\section{Introduction and Notation}

 We consider the category of commutative and unital rings. An epimorphism  is an epimorphism of  this  category. Let $R\subseteq S$ be a (ring) extension. The set of all $R$-subalgebras of $S$ is denoted by $[R,S]$. The extension $R\subseteq S$ is said to have FIP (for the ``finitely many intermediate algebras property") if $[R,S]$ is finite.   A {\it chain} of $R$-subalgebras of $S$ is a set of elements of $[R, S] $ that are pairwise comparable with respect to inclusion. We say that the extension $R\subseteq S$ has FCP (for the ``finite chain property") if each chain in $[R,S]$ is finite.  Dobbs and the authors characterized FCP and FIP extensions  \cite{DPP2}. Clearly, an  extension that satisfies FIP must also satisfy FCP. An extension $R\subseteq S$ is called FMC if there is a finite maximal chain of extensions from $R$ to $S$.
 
We begin  by explaining our motivations and aims. The reader who is not familiar with the notions used will find some Scholia in the sequel, as well as necessary definitions that exist in the literature. Knebusch and Zang introduced Pr\"ufer extensions in their book \cite{KZ}. Actually, these extensions are nothing but  normal pairs, that are intensively studied in the literature. We do not intend to give an extensive list of recent papers, written by Ayache, Ben Nasr, Dobbs, Jaballah, Jarboui and some others. We are indebted to these authors because their papers  are a rich source of suggestions. We observed that some of them are dealing with FCP (FIP, FMC) extensions, followed by a Pr\"ufer extension, perhaps under a hidden form. These extensions reminded us quasi-Pr\"ufer domains (see \cite{FHP}).  Therefore, we   introduced in \cite{P1} {\em quasi-Pr\"ufer} extensions $R\subseteq S$ as extensions that can be factored $R\subseteq R' \subseteq S$, where the first extension is integral and the second is Pr\"ufer.  Note that FMC extensions are quasi-Pr\"ufer. 

We give a systematic study of quasi-Pr\"ufer extensions in Section 2 and Section 3. The class of quasi-Pr\"ufer extensions has a nice behavior  with respect to  the classical operations of commutative algebra. An important result is that quasi-Pr\"ufer extensions coincide with INC-pairs.  Another one is that this class is stable under forming subextensions and  composition. A striking result is the stability of the class  of quasi-Pr\"ufer extensions by absolutely flat  base change, like localizations and Henselizations. Any ring extension $R\subseteq S$ admits a quasi-Pr\"ufer closure, contained in $S$. Examples are provided by Laskerian pairs, open pairs and the pseudo-Pr\"ufer pairs of Dobbs-Shapiro \cite{DS}.

 Section 4 deals with {\em almost-Pr\"ufer} extensions, a special kind of quasi-Pr\"ufer extensions. They are of the form $R\subseteq T \subseteq S$, where the first extension is Pr\"ufer and the second is integral. Any ring extension admits an almost-Pr\"ufer closure, contained in $S$. The class of almost-Pr\"ufer extensions seems to have less properties than  the class of quasi-Pr\"ufer extensions but has the advantage of the commutation of Pr\"ufer closures with localizations at prime ideals.
 We examine the transfer of the quasi (almost)-Pr\"ufer properties to subextensions.
 
 Section 5  study the transfer of the quasi (almost)-Pr\"ufer properties to Nagata extensions.
 
 In section 6, we complete and generalize the results of Ayache-Dobbs in \cite{AD}, with respect to the finiteness of fibers. These authors have evidently considered particular cases of quasi-Pr\"ufer extensions. A main result is that if $R\subseteq S$ is quasi-Pr\"ufer with finite fibers, then so is $R\subseteq T$ for $T\in [R,S]$. In particular, we recover a result of \cite{AD} about FMC extensions.
 
 Now Section 7 gives  calculations of $|[R,S]|$ with respect to its Pr\"ufer closure, quasi-Pr\"ufer (almost-Pr\"ufer) closure  in case $R\subseteq S$ has FCP.
 
 \subsection{ Recalls about some results and definitions}
 
 The reader is warned that we will mostly use the definition of Pr\"ufer extensions by flat epimorphic subextensions investigated in \cite{KZ}. The results needed may be found in Scholium A for flat epimorphic extensions and some results of \cite{KZ}  are summarized  in Scholium B. Their powers give quick  proofs of results that are generalizations of results of the literature.

As long as FCP or FMC extensions are concerned,  we use  minimal (ring) extensions, a concept  introduced by Ferrand-Olivier \cite{FO}. An extension $R\subset S$ is called {\it minimal} if $[R,S]=\{R,S\}$. It is known that a minimal extension is either module-finite or a flat epimorphism \cite{FO} and these conditions are mutually exclusive. There are three types of integral minimal (module-finite) extensions: ramified, decomposed or inert \cite[Theorem 3.3]{Pic}. A minimal extension  $R\subset S$ admits a crucial ideal $\mathcal{ C}(R,S)= : M$ which is maximal in $R$ and such that $R_P=S_P$ for each $P\neq M, P\in \mathrm{Spec}(R)$. Moreover, $\mathcal{ C}(R,S)=(R:S)$ when $R\subset S$ is an integral minimal extension. 
 The key connection between the above ideas is that if $R\subseteq S$ has FCP or FMC, then any maximal (necessarily finite) chain of $R$-subalgebras of $S$, $R=R_0\subset R_1\subset\cdots\subset R_{n- 1}\subset R_n=S$, with {\it length} $n<\infty$, results from juxtaposing $n$ minimal extensions $R_i\subset R_{i+1},\ 0\leq i\leq n-1$. 
 
 Following \cite{J}, we define the {\it length} $\ell[R,S]$ of $[R,S]$ as  the supremum of the lengths of chains in $[R,S]$. In particular, if $\ell[R,S]=r$, for some integer $r$, there exists a maximal chain in $[R,S]$  with length $r$. 
 
As usual, Spec$(R)$,  Max$(R)$, Min$(R)$, U$(R)$, Tot$(R)$ are respectively the set of prime ideals,  maximal ideals, minimal prime ideals, units, total ring of fractions of a ring $R$ and $\kappa (P) = R_P/PR_P$ is the residual field of $R$ at $P\in \mathrm{Spec}(R)$. 

If $R\subseteq S$ is an extension, then $(R:S)$ is its conductor and if $P\in\mathrm{Spec}(R)$, then $S_P$ is the localization $S_{R\setminus P}$. We denote the integral closure of $R$ in $S$ by $\overline R^S$ (or $\overline R$). 
 
 A local ring is here what is called elsewhere a quasi-local ring. The {\it  support} of an $R$-module $E$ is $\mathrm{Supp}_R(E):= \{P\in\mathrm {Spec}(R)\mid  E_P\neq 0\}$ and $\mathrm{MSupp}_R(E):=\mathrm{Supp}_R(E)\cap\mathrm{Max}(R)$.  
Finally, $\subset$ denotes proper inclusion and $|X|$  the cardinality of a set $X$.

{\bf Scholium A}
We give some recalls about flat epimorphisms (see \cite[Chapitre IV]{L}, except (2) which is \cite[Proposition 2]{OEPI}).
\begin{enumerate}
\item    $R\to S$ is a flat epimorphism $\Leftrightarrow$  for all $P\in \mathrm{Spec}(R)$, either $R_P\to S_P$ is an isomorphism or $S=PS$ $\Leftrightarrow$  $R_P\subseteq S_P$ is a flat epimorphism for all $P\in \mathrm{Spec}(R)$.

\item (S) A flat epimorphism, with a zero-dimensional domain, is surjective.

\item If $f:A \to B$ and $g: B\to C$ are ring morphisms such that $g\circ f$ is injective and $f$ is a flat epimorphism, then $g$ is injective.

\item Let  $R\subseteq T \subseteq S$ be a tower of extensions, such that $R\subseteq  S$ is a flat epimorphism. Then $T\subseteq S$ is a flat epimorphism but $R\subseteq T$ does not need. A Pr\"ufer extension remedies to this defect.

\item (L) A faithfully flat epimorphism is an isomorphism. Hence, $R= S$ if  $R\subseteq S$  is an integral flat epimorphism.

\item  If $f: R\to S$ is a flat epimorphism and $J$ an ideal of $S$, then $J = f^{-1}(J)S$.

\item If $f: R \to S$ is an epimorphism, then $f$ is spectrally injective and its residual extensions are isomorphisms.

\item Flat epimorphisms remain flat epimorphisms under base change (in particular, after a localization with respect  to a multiplicatively closed subset).

\item Flat epimorphisms are descended  by faithfully flat morphisms.
\end{enumerate}

\subsection{Recalls and results on Pr\"ufer extensions}

We recall some definitions  and properties of ring extensions $R\subseteq S$ and rings $R$. There are a lot of characterizations of Pr\"ufer extensions. We keep only
those that are useful in this paper. We give  the two definitions that are dual and emphasize some characterizations in the local case.

{\bf  Scholium B}

\begin{enumerate}
\item \cite{KZ} $R\subseteq S$ is called Pr\"ufer if $R\subseteq T$ is a flat epimorphism  for each $T\in [R,S]$.

\item   $R\subseteq S$ is called a {\it normal} pair if $T\subseteq S$ is integrally closed for each $T\in [R,S]$.

\item $R\subseteq S$ is Pr\"ufer if and only if it is a normal pair 
\cite[Theorem 5.2(4)]{KZ}.

\item $R$  is called Pr\"ufer if  its finitely generated regular ideals are invertible, or equivalently, $R\subseteq \mathrm{Tot}(R)$  is Pr\"ufer \cite[Theorem 13((5)(9))]{GR}.
\end{enumerate}

Hence  Pr\"ufer extensions are a relativization of Pr\"ufer rings.  Clearly, a minimal extension is a flat epimorphism if and only if it is Pr\"ufer. We will then use for such extensions the terminology: {\it Pr\"ufer minimal} extensions. The reader may find  some properties of Pr\"ufer minimal extensions in \cite[Proposition 3.2, Lemma 3.4 and Proposition 3.5]{Pic}, asserted by L. Dechene in her dissertation, but where in addition  $R$ must be supposed local. The reason why is that this word has surprisingly disappeared during the printing process of \cite{Pic}.

We will need the  two next  results. Some of them  do not  explicitely appear in \cite{KZ} but deserve to be emphasized. We refer to \cite[Definition 1, p.22]{KZ} for a definition of Manis extensions.

\begin{proposition}\label{0.2} Let $R\subseteq S$ be a ring extension.
 \begin{enumerate}
  
  \item $R\subseteq S$ is Pr\"ufer if and only if $R_P\subseteq S_P$ is Pr\"ufer for each $P\in \mathrm{Spec}(R)$ (respectively, $P\in \mathrm{Supp}(S/R)$).
 
  \item  $R \subseteq S$ is Pr\"ufer if and only if $R_M\subseteq S_M$ is Manis for each $M\in \mathrm{Max}(R)$.
  \end{enumerate}
  \end{proposition}
  \begin{proof} (1) The class of Pr\"ufer extensions is stable under localization \cite[Proposition 5.1(ii), p.46-47]{KZ}. To get the converse, use Scholium A(1).  (2) follows from \cite[Proposition 2.10, p.28, Definition 1, p.46]{KZ}.
 \end{proof}

  \begin{proposition}\label{0.3} Let $R\subseteq S$ be a ring extension, where $R$ is local. 
  \begin{enumerate}
  \item  $R\subseteq S$ is Manis if and only if $S\setminus R \subseteq \mathrm{U}(S)$ and  $x \in S\setminus R \Rightarrow x^{-1} \in R$. In that case, $R\subseteq S$ is integrally closed.
  \item  $R\subseteq S$ is Manis if and  only if $R\subseteq S$ is Pr\"ufer.
  
  \item   $R\subseteq S$ is Pr\"ufer if and only if there exists $P\in \mathrm{Spec}(R)$ such that $S=R_P$, $P=SP$ and $R/P$ is a valuation domain. Under these conditions, $S/P$ is the quotient field of $R/P$.
    \end{enumerate}
  \end{proposition}
  \begin{proof}  (1) is \cite[Theorem 2.5, p.24]{KZ}. 
  (2) is \cite[Scholium 10.4, p.147]{KZ}.
  Then (3) is \cite[Theorem 6.8]{DPP2}. 
   \end{proof}
   
     Next result shows that Pr\"ufer FCP extensions  can be described  in a special manner.
   
   \begin{proposition}\label{0.5} Let $R\subset S$ be a ring extension. 
\begin{enumerate}
\item If $R\subset S$ has FCP, then $R\subset S$ is integrally closed $\Leftrightarrow$  $R\subset S$ is  Pr\"ufer $\Leftrightarrow$ $R\subset S$ is a composite of Pr\"ufer minimal extensions.

\item If $R\subset S$ is integrally closed, then $R\subset S$ has FCP $\Leftrightarrow$ $R\subset S$ is  Pr\"ufer and   $\mathrm{Supp}(S/R)$ is finite. 
\end{enumerate}
\end{proposition}

\begin{proof} (1) Assume that $R\subset S$ has FCP. If $R\subset S$ is integrally closed, then, $R\subset S$ is composed of Pr\"ufer minimal extensions by \cite[Lemma 3.10]{DPP2}. Conversely, if $R\subset S$ is composed of Pr\"ufer minimal extensions,  $R\subset S$ is integrally closed, since so is each Pr\"ufer minimal extension. 
A Pr\"ufer extension is obviously integrally closed, and an FCP integrally closed extension is Pr\"ufer by \cite[Theorem 6.3]{DPP2}.

(2) The logical  equivalence is \cite[Theorem 6.3]{DPP2}. 
\end{proof}

\begin{definition}\label{0.6} \cite{KZ} A ring extension $R\subseteq S$ has:
\begin{enumerate}
\item  a greatest flat epimorphic subextension  $R \subseteq \widehat R^S$, called the {\bf Morita hull} of $R$ in  $S$. 
\item  a  greatest Pr\"ufer subextension $R \subseteq \widetilde R^S$, called the {\bf Pr\"ufer hull}  of $R$ in $S$. 

\end{enumerate}
 We set  $\widehat{R} := \widehat R^S$ and $\widetilde{R} := \widetilde R^S$, if no confusion can occur.
$R\subseteq S$ is called Pr\"ufer-closed if $R= \widetilde R$.
 \end{definition}

 Note that $\widetilde R^S$ is denoted by $\mathrm P(R,S)$ in \cite{KZ} and $\widehat R^S$ is the weakly surjective hull $\mathrm{M}(R,S)$ of \cite{KZ}. Our terminology is justified because Morita's work is earlier \cite[Corollary 3.4]{M}.  The Morita hull can be computed by using a (transfinite) induction \cite{M}. Let $S'$ be the set of all  $s\in S$  such that there is some ideal $I$ of $R$, such that $IS= S$ and $Is \subseteq R$. 
Then $R\subseteq S'$ is a subextension of $R\subseteq S$. We set $S_1:= S'$ and  $S_{i+1}:= (S_i)' \subseteq S_i$. By \cite[p.36]{M}, if $R \subset S$ is an FCP extension, then $\widehat R = S_n$ for some integer $n$.

 At this stage it is interesting to point out a result;
  showing again that integral closedness and Pr\"ufer extensions  are closely related. 
 
\begin{proposition}\label{0.7}  Olivier \cite[Corollary, p.56]{O} An extension $R\subseteq S$ is integrally closed if and only if there is a pullback square:
\car R,S,V,K  

\noindent  where $V$ is a semi-hereditary ring and $K$ its total quotient ring. 
\end{proposition}
In that case  $V\subseteq K$ is a Pr\"ufer extension, since $V$ is a Pr\"ufer ring, whose localizations at prime ideals are valuation domains and $K$ is an absolutely flat ring. As there exist integrally closed extensions that are not Pr\"ufer, we see in passing that the pullback construction may not descend Pr\"ufer extensions. The above result has a companion  for minimal extensions that are Pr\"ufer \cite[Proposition 3.2]{GRA}.

\begin{proposition}\label{0.8} Let $R\subseteq S$ be an extension and $T \in [R,S]$, then $\widetilde R ^T = \widetilde R \cap T$. Therefore, for  $T, U\in [R,S]$ with $T\subseteq U$, then $\widetilde R^T \subseteq \widetilde R^U$.
\end{proposition}
\begin{proof} Obvious, since the Pr\"ufer hull  $\widetilde R^ T$ is the greatest Pr\"ufer extension $R \subseteq V$ contained in $T$.
\end{proof}

We will show later that in some cases $\widetilde T\subseteq \widetilde U$ if $R\subseteq S$ has FCP.

\section{Quasi-Pr\"ufer extensions}

We introduced the following definition  in \cite[p.10]{P1}. 

\begin{definition}\label{2.1} An extension of rings $R\subseteq S$ is called quasi-Pr\"ufer if one of the following equivalent statements holds:

\begin{enumerate}

\item   $\overline R \subseteq S$ is a Pr\"ufer extension;

\item $R\subseteq S$ can be factored $R\subseteq T \subseteq S$, where $R\subseteq T$ is integral and $T\subseteq S$ is Pr\"ufer. In that case $\overline R = T$

\end{enumerate}
\end{definition}
 To see that (2) $\Rightarrow$ (1) observe that if (2) holds, then $T\subseteq \overline R$ is integral and a flat injective epimorphism, so that $\overline R= T$ by  (L) (Scholium A(5)).

 We observe that quasi-Pr\"ufer extensions are akin to quasi-finite extensions if we refer to Zariski Main  Theorem. This will be explored in Section 6, see for example Theorem~\ref{6.0}.

Hence integral or Pr\"ufer extensions are quasi-Pr\"ufer.   An extension is  clearly Pr\"ufer if and only if it  is quasi-Pr\"ufer and integrally closed. Quasi-Pr\"ufer extensions allow us to avoid FCP hypotheses.

We give some other definitions involved  in  ring extensions $R\subseteq S$.
The{ \it fiber} at $P\in \mathrm{Spec}(R)$ of $R\subseteq S$ is $\mathrm{Fib}_{R,S}(P):= \{ Q\in \mathrm{Spec}(S) \mid Q\cap R = P\}$.  The subspace $\mathrm{Fib}_{R,S}(P)$ of $\mathrm{Spec}(S)$ is homeomorphic  to the spectrum of the fiber ring  $\mathrm{F}_{R,S}(P):=\kappa(P)\otimes_R S$  at $P$. The homeomorphism is given by the spectral map of  $S \to\kappa(P)\otimes_R S$ and $\kappa(P) \to\kappa(P) \otimes_RS$ is the {\it fiber morphism} at $P$. 
\begin{definition}\label{2.2} A ring extension $R\subseteq S$ is called:

\begin{enumerate}
\item  {\it incomparable} if for each pair $Q\subseteq Q'$ of prime ideals of $S$, then $Q\cap R =Q'\cap R\Rightarrow Q= Q'$, or equivalently, $\kappa(P)\otimes_RT$ is a zero-dimensional ring for each $T\in [R,S]$ and $P\in \mathrm{Spec}(R)$, such that $\kappa(P)\otimes_RT \neq 0$. 

\item  an {\it INC-pair} if $R\subseteq T$ is incomparable for each $T\in [R,S]$.   

 \item  {\it residually algebraic} if $R/(Q\cap R) \subseteq S/Q$ is algebraic for each $Q\in \mathrm{Spec}(S)$.
 
 \item a  {\it residually algebraic pair} if  the extension $R\subseteq T$ is residually algebraic for each $T\in [R,S]$.

\end{enumerate}
\end{definition}

The following characterization  was  announced in \cite{P1}. We were unaware that this result is also proved in \cite[Corollary 1]{BJ}, when we present it in ArXiv. However, our proof is largely shorter because we use the powerful results of \cite{KZ}.

\begin{theorem}\label{2.3} An extension $R\subseteq S$ is quasi-Pr\"ufer if and only if $R\subseteq S$ is an INC-pair and, if and only if, $R\subseteq S$ is a residually algebraic pair.
\end{theorem}
\begin{proof} Suppose that $R\subseteq S$ is quasi-Pr\"ufer and let $T\in [R,S]$. We set $U:= \overline RT$. Then $\overline R\subseteq U$ is a flat epimorphism by definition of a Pr\"ufer extension and hence is incomparable as is $R\subseteq \overline  R$ . It follows that $R\subseteq U$ is incomparable. Since $T\subseteq U$ is integral, it has going-up. It follows that $R\subseteq T$ is incomparable. Conversely, if $R\subseteq S$ is an INC-pair, then so is $\overline R \subseteq S$. Since $\overline R \subseteq S$ is integrally closed, $\overline R \subseteq S$ is  Pr\"ufer \cite[Theorem 5.2,(9'), p.48]{KZ}. The second equivalence is \cite[Proposition 2.1]{D} or \cite[Theorem 6.5.6]{FHP}.
\end{proof}

 \begin{corollary} An extension $R\subseteq S$ is quasi-Pr\"ufer if and only if $\overline R \subseteq \overline T$ is Pr\"ufer for each $T \in [R,S]$.
 
 \end{corollary}

It follows that most of the properties described in \cite{AJ} for integrally closed INC-pairs of domains are valid for arbitrary ring extensions. Moreover, a result of Dobbs is easily gotten: an INC-pair $R\subseteq S$ is an integral extension if and only if $\overline R \subseteq S$ is spectrally surjective \cite[Theorem 2.2]{D}. This follows from Scholium A, Property (L).

\begin{example}\label{2.5}

Quasi-Pr\"ufer domains $R$ with quotient fields $K$ can be characterized by $R\subseteq K$ is quasi-Pr\"ufer.  The reader may  consult \cite[Theorem 1.1]{CF} or \cite{FHP}. In view of \cite[Theorem 2.7]{ADF}, $R$ is a quasi-Pr\"ufer domain if and only if $\mathrm{Spec}(R(X)) \to \mathrm{Spec}(R)$ is  
bijective. 

 We give here another  example of quasi-Pr\"ufer extension. An  extension $R\subset S$  is called a {\it  going-down pair} if  each  of its subextensions has the going-down property.  For such a pair,  $R\subseteq T$ has incomparability for  each $T\in [R,S]$, at each  non-maximal prime ideal of $R$ \cite[Lemma 5.8]{ABEJ}(ii).  Now let $M$ be a maximal ideal of $R$, whose fiber is not void in $T$. Then $R\subseteq T$ is a going-down pair, and so is $R/M \subseteq T/MT$ because $MT\cap R= M$. By \cite[Corollary 5.6]{ABEJ}, the dimension of $T/MT$ is $\leq 1$. Therefore, if $R\subset S$  is a going-down pair, then $R\subset S$ is quasi-Pr\"ufer if and only if  $\mathrm{dim}(T/MT)\neq 1$ for each $T\in [R,S]$ and $M\in \mathrm{Max}(R)$.
 
 Also {\it open-ring pairs} $R\subset S$  are quasi-Pr\"ufer  by \cite[Proposition 2.13]{C}.
 
An  $i$-{\it pair} is an extension $R\subseteq S$ such that $\mathrm{Spec}(T) \to \mathrm{Spec}(R)$ is injective for each $T\in [R,S]$, or equivalently if and only if $R \subseteq S$ is quasi-Pr\"ufer and $R \subseteq \overline R$ is spectrally injective \cite[Proposition 5.8]{P1}. These extensions appear frequently in the integral domains context.  Another examples are  given by some  extensions $R\subseteq S$, such that   $\mathrm{Spec}(S)=\mathrm{Spec}(R)$ as sets, as we will see later.
\end{example}

\section{Properties of   quasi-Pr\"ufer extensions}

We now develop the machinery of quasi-Pr\"ufer extensions.

\begin{proposition}\label{3.1} An extension $R\subset S$  is  (quasi-)Pr\"ufer if and only if $R_P\subseteq S_P$ is  (quasi-)Pr\"ufer  for any  $P\in \mathrm{Spec}(R)$ ($P\in\mathrm{MSupp}(S/R)$).
\end{proposition}

\begin{proof} The proof is easy if we use the INC-pair property definition of
 quasi-Pr\"ufer extension (see also \cite[Proposition 2.4]{AJ}).
\end{proof}

\begin{proposition}\label{3.2} Let $R\subseteq S$ be a quasi-Pr\"ufer  extension and $\varphi : S \to S'$ an integral ring morphism. Then $\varphi(R) \subseteq S'$ is  quasi-Pr\"ufer  and $S' = \varphi(S)\overline{\varphi(\overline R)}$, where $\overline{\varphi(\overline R)}$ is the integral closure of $\varphi(\overline R)$ in $S'$. 
\end{proposition}
\begin{proof} It is enough to apply \cite[Theorem 5.9]{KZ} to the Pr\"ufer extension
$\overline R \subseteq S$ and to use Definition~\ref{2.1}.
\end{proof}

This result applies  with $S':= S\otimes_R R'$, where $R\to R'$ is an integral morphism. Therefore integrality ascends the quasi-Pr\"ufer property.

We know that a composite of Pr\"ufer extensions is  Pr\"ufer  \cite[Theorem 5.6, p.51]{KZ}. The following Corollary~\ref{3.3} contains \cite[Theorem 3]{BJ}.

\begin{corollary}\label{3.3} Let $R\subseteq T \subseteq S$ be a tower of extensions. Then $R\subseteq S$ is quasi-Pr\"ufer if and only if $R\subseteq T$ and $T\subseteq S$ are quasi-Pr\"ufer. It follows that $R\subseteq T$ is quasi-Pr\"ufer if and only if $R\subseteq \overline RT$ is quasi-Pr\"ufer.
\end{corollary}
\begin{proof}   
Consider a tower ($\mathcal T$) of extensions $R \subseteq \overline R \subseteq S:= R' \subseteq \overline{R'} \subseteq S'$ (a composite of two quasi-Pr\"ufer extensions). By using Proposition~\ref{3.2} we see that $\overline R \subseteq S= R' \subseteq \overline{R'}$ is quasi-Pr\"ufer. Then ($\mathcal T$) is obtained by writing on the left an integral extension and on the right a Pr\"ufer extension. Therefore, ($\mathcal T$) is  quasi-Pr\"ufer. We prove the converse.

 If $R\subseteq T \subseteq  S$ is  a tower of extensions, then $R\subseteq T$ and $T\subseteq S$ are INC-pairs whenever $R\subseteq S$ is an INC-pair.  The converse is then a consequence of Theorem~\ref{2.3}.
 
 The last statement is \cite[Corollary 4]{BJ}.
\end{proof}

Using the above corollary, we can exhibit  new examples of quasi-Pr\"ufer extensions. We recall that a ring $R$ is called {\it Laskerian}  if each of its ideals is a finite intersection of primary ideals and a ring extension $R\subset S$  a {\it Laskerian pair} if each $T\in [R,S]$ is a Laskerian ring. Then \cite[Proposition 2.1]{V} shows that if $R$ is an integral domain  with quotient field  $F\neq R$ and $F\subset K$ is a field extension, then $R\subset K$ is a Laskerian pair if and only if $K$ is algebraic over $R$ and $\overline R$ (in $K$)  is a Laskerian Pr\"ufer domain. It follows easily that $R\subset K$ is quasi-Pr\"ufer.

Next result generalizes \cite[Proposition 1]{JM}.

\begin{corollary}~\label{3.4} An FMC extension $R\subset S$  is quasi-Pr\"ufer.
\end{corollary}
\begin{proof} Because $R\subset S$ is a composite of finitely many minimal extensions, by Corollary~\ref{3.3}, it is enough to observe that a minimal extension   is either  Pr\"ufer or integral.
\end{proof}

\begin{corollary}\label{3.5} Let $R\subseteq S$ be a quasi-Pr\"ufer extension and a tower $R\subseteq T \subseteq S$, where $R\subseteq T$ is integrally closed. Then $R\subseteq T$ is Pr\"ufer.
\end{corollary}
\begin{proof} Observe that $R\subseteq T$ is quasi-Pr\"ufer and then that $R= \overline R^T$.
\end{proof}

Next result deals with the Dobbs-Shapiro {\it  pseudo-Pr\"ufer} extensions of integral domains  \cite{DS}, that they  called pseudo-normal pairs.  Suppose that $R$ is local, we call here pseudo-Pr\"ufer an extension $R\subseteq  S$ such that there exists $T\in[R,S]$ with $\mathrm{Spec}(R) =\mathrm{Spec}(T)$ and $T \subseteq S$ is Pr\"ufer \cite[Corollary 2.5]{DS}. If $R$ is arbitrary, the extension $R\subseteq S$ is called pseudo-Pr\"ufer if $R_M \subseteq S_M$ is pseudo-Pr\"ufer for each $M\in \mathrm{Max}(R)$. In view of the Corollary~\ref{3.3}, it is enough to characterize quasi-Pr\"ufer extensions of the type $R\subseteq T$ with  $\mathrm{Spec}(R) =\mathrm{Spec}(T)$.

\begin{corollary}\label{3.6} Let $R\subseteq T$ be an extension with  $\mathrm{Spec}(R) =\mathrm{Spec}(T)$ and $(R,M)$  local. Then $R\subseteq T$ is quasi-Pr\"ufer if and only if  $\mathrm{Spec}(R) =\mathrm{Spec}(U)$ for each $U\in [R,T]$ and, if and only if $R/M \subseteq T/M$ is an algebraic field extension. In such a case, $R\subseteq T$ is Pr\"ufer-closed.
\end{corollary}
\begin{proof} It follows from \cite{ADo} that $M\in \mathrm{Max}(T)$. Part of the proof is gotten by observing that $R\subseteq U$ is an INC extension  if   $\mathrm{Spec}(R) =\mathrm{Spec}(U)$. Another one is proved in \cite[Corollary 3.26]{ADo}. Now  $R\subseteq \widetilde R$ is a  spectrally surjective flat epimorphism  and then, by Scholium A, $R= \widetilde R$.
\end{proof}

Let $R\subseteq S$ be an extension and  $I$  an ideal shared with $R$ and $S$. It is easy to show that $R\subseteq S$ is quasi-Pr\"ufer if and only if $R/I\subseteq S/I$ is quasi-Pr\"ufer by using \cite[Proposition 5.8]{KZ} in the Pr\"ufer case. We are able to give a more general statement.

\begin{lemma}\label{3.7} Let $ R\subseteq S$ be a (quasi-)Pr\"ufer extension and $J$ an ideal of $S$ with $I= J\cap R$. Then $R/I \subseteq S/J$ is a (quasi-)Pr\"ufer extension. If $R\subseteq S$ is Pr\"ufer and  $N$  is  a maximal ideal of $S$, then $R/(N\cap R)$ is a valuation domain with quotient field $S/N$.
\end{lemma}

\begin{proof} Assume first that $ R\subseteq S$ is Pr\"ufer. We have $J = IS$ by Scholium A(6), because $R \subseteq S$ is a flat epimorpism. Therefore, any  $D\in [R/I,S/J]$ is of the form $C/J$ where $C \in [R,S]$. We can write $C/IS = (C+I)/IS \cong  C/C\cap IS$. As $R \subseteq C$ is a flat epimorphism,  $C \cap IS =IC$. It then follows that $D= C\otimes R/I$ and we get easily that $R/I\subseteq S/J$ is  Pr\"ufer, since $R/I \subseteq D$ is a flat epimorphism. The quasi-Pr\"ufer case is an easy consequence. 
\end{proof}

With this lemma we generalize and complete \cite[Proposition 1.1]{H}.

\begin{proposition}\label{3.8} Let $R\subseteq S$ be an extension of rings. The following statements are equivalent:
\begin{enumerate}
\item $R\subseteq S$ is quasi-Pr\"ufer;
\item $R/(Q\cap R) \subseteq S/Q$ is quasi-Pr\"ufer for each  $Q\in \mathrm{Spec}(S)$ ;
\item $(X-s)S[X]\cap R[X] \not\subseteq M[X]$ for each $s\in S$ and $M\in \mathrm{Max}(R)$;
\item For each $T\in [R,S]$, the fiber morphisms of $R\subseteq T$ are integral.
\end{enumerate}
\end{proposition}
\begin{proof} (1) $\Rightarrow$ (2) is entailed by Lemma~\ref{3.7}. Assume that (2) holds and let $M\in \mathrm{Max}(R)$ that contains a minimal prime ideal $P$, lain over by a minimal prime ideal $Q$ of $S$.  Then (2) $\Rightarrow$ (3) follows from \cite[Proposition 1.1(1)]{H}, applied to $R/(Q\cap R)\subseteq S/Q$. If (3) holds,  argue as in the paragraph before \cite[Proposition 1.1]{H} to get that $R\subseteq S$ is a $\mathcal P$-extension, whence an INC-extension by \cite[Proposition 2.1]{D}. Because  integral extensions have incomparability, we see that (4) $\Rightarrow$ (1). Corollary~\ref{3.3} shows that the reverse implication holds, if any quasi-Pr\"ufer extension  $R\subseteq S$ has integral fiber morphisms. 
For $P\in \mathrm{Spec}(R)$, the extension $R_P/PR_P \subseteq S_P/PS_P$ is quasi-Pr\"ufer by Lemma~\ref{3.7}. The ring $\overline R_P/P\overline R_P$ is zero-dimensional and  $\overline R_P/P\overline R_P \to S_P/PS_P$, being a flat epimorphism, is therefore surjective by Scholium A (S). It follows that the fiber morphism at $P$ is integral.
\end{proof}

\begin{remark}\label{3.9} The logical equivalence (1) $\Leftrightarrow$ (2) is still valid if we replace quasi-Pr\"ufer with integral in the above proposition. It is enough to show that  an extension $R\subseteq S$ is  integral when $R/P \subseteq S/Q$ is integral for each $Q\in \mathrm{Spec}(S)$ and $P:= Q\cap R$. We can suppose that $ S= R[s] \cong R[X]/I$, where $X$ is an indeterminate, $I$ an ideal of $R[X]$ and  $Q$ varies in Min$(S)$, because for an extension $A\subseteq B$, any element  of Min($A$) is lain over by some element of Min($B$). If $\Sigma$ is the set of unitary polynomials of $R[X]$, the assumptions show that any element of $\mathrm{Spec}(R[X])$, containing $I$, meets $\Sigma$. As $\Sigma$ is a multiplicatively closed subset, $I\cap \Sigma \neq \emptyset$, whence $s$ is integral over $R$.

But a similar result does not hold if we replace quasi-Pr\"ufer with Pr\"ufer, except if we suppose that $R\subseteq S$ is integrally closed. To see this, apply the above proposition to get a quasi-Pr\"ufer extension $R\subseteq S$ if each $R/P\subseteq S/Q$ is Pr\"ufer. Actually, this situation already occurs for Pr\"ufer rings and their factor domains, as Lucas's paper \cite{LU} shows. More precisely, \cite[Proposition 2.7]{LU}  and the third paragraph of \cite[p. 336]{LU} shows that  if $R$ is a ring with $\mathrm{Tot}(R)$ absolutely flat, then $R$ is a quasi-Pr\"ufer ring if  $R/P$ is a Pr\"ufer domain for each $P\in \mathrm{Spec}(R)$. Now example \cite[Example 2.4]{LU} shows that $R$ is not necessarily Pr\"ufer.

\end{remark}

We observe that if $R\subseteq S$ is  quasi-Pr\"ufer, then $R/M$ is a quasi-Pr\"ufer domain for each $N\in \mathrm{Max}(S)$ and $M: =N\cap R$   (in case $R\subseteq S$ is integral, $R/M$ is a field). To prove this, observe that $R/M\subseteq S/N$ can be factored $R/M\subseteq \kappa(M) \subseteq  S/N$. As we will see, $R/M\subseteq \kappa(M)$ is quasi-Pr\"ufer because $R/M\subseteq S/N$ is quasi-Pr\"ufer.

The class of Pr\"ufer extensions is not stable by (flat) base change. For example, let $V$ be a valuation domain with quotient field $K$. Then $V[X] \subseteq K[X]$ is not Pr\"ufer  \cite[Example 5.12, p.53]{KZ}. Thus if we consider an ideal $I$ of $R$ and $J:= IS$, $R \subseteq S$ Pr\"ufer may not imply $R/I \subseteq S/IS$ Pr\"ufer except if $IS \cap R = I$. This happens for instance for a prime ideal $I$  of $R$ that is lain over by a prime ideal of $S$.

\begin{proposition}\label{3.10} Let $R\subseteq S$ be  a (quasi)-Pr\"ufer extension and  $R\to T$ a flat epimorphism, then $T\subseteq S\otimes_RT$ is (quasi)-Pr\"ufer. If in addition   $S$ and $T$ are both subrings of some ring and $R\subseteq T$ is an extension,  then $T\subseteq TS$ is (quasi)-Pr\"ufer.
\end{proposition}
\begin{proof} For the first part, it is enough to consider the Pr\"ufer case.  It is well known that the following diagram is a pushout if $Q\in \mathrm{Spec}(T)$ is lying over $P$ in $R$:

 \car {R_P},{S_P},{T_Q},{(T\otimes_RS)_Q}
 
 As $R_P \to T_Q$ is an isomorphism since $R\to T$ is a flat epimorphism by Scholium A, it follows that $R_P \subseteq S_P$ identifies to $T_Q \to (T\otimes_R S)_Q$. The result follows  because Pr\"ufer extensions localize and globalize. 
 
 In case $R\to T$ is a flat epimorphic extension, the surjective maps $T\otimes_RS \to TS$ and $\overline R\otimes_RT \to \overline RT$ are  isomorphisms because $\overline R \to T\overline R$ (resp. $S \to ST$) is injective and $\overline R\to T\otimes_T\overline R$ (resp. $S \to S\otimes_R T$) is a flat epimorphism. Then it is enough to use  Scholium A.
\end{proof} 

The reader may find in \cite[Corollary 5.11, p.53]{KZ} that if $R\subseteq A \subseteq S$ and $R\subseteq B \subseteq S$ are extensions  and $R\subseteq A$ and  $R\subseteq B$ are both Pr\"ufer, then $R \subseteq AB$ is Pr\"ufer.

\begin{proposition}\label{3.11} Let $R\subseteq A$ and $R\subseteq B$ be two extensions, where $A$ and $B$ are subrings of a ring $S$.  If they are both quasi-Pr\"ufer, then $R \subseteq AB$ is quasi-Pr\"ufer.
\end{proposition}
\begin{proof} Let $U$ and $V$ be the integral closures of $R$ in $A$ and $B$. Then $R \subseteq A \subseteq AV$ is quasi-Pr\"ufer because $A\subseteq AV$ is integral and Corollary~\ref{3.3} applies. Using  again Corollary~\ref{3.3} with $R\subseteq V \subseteq AV$, we find that $V\subseteq AV$ is quasi-Pr\"ufer. Now Proposition~\ref{3.10} entails that $B\subseteq AB$ is quasi-Pr\"ufer because $V \subseteq B$ is a flat epimorphism. Finally $R \subseteq AB$ is quasi-Pr\"ufer, since a composite of quasi-Pr\"ufer extensions.
\end{proof}

It is known that an arbitrary  product of extensions is  Pr\"ufer if and only if each of its components is ¬Pr\"ufer \cite[Proposition 5.20, p.56]{KZ}. The following result is an easy consequence.

\begin{proposition}\label{3.12} Let $\{ R_i \subseteq S_i | i=1,\ldots, n\}$ be a  finite family of quasi-Pr\"ufer extensions, then $R_1\times \cdots \times R_n \subseteq S_1\times \cdots \times S_n$ is quasi-Pr\"ufer. In particular, if $\{ R \subseteq S_i | i=1,\ldots, n\}$ is a finite family of quasi-Pr\"ufer extensions, then $R \subseteq S_1\times \cdots \times S_n$ is quasi-Pr\"ufer.
\end{proposition}
In the same way we have  the following result deduced from \cite[Remark 5.14, p.54]{KZ}.

\begin{proposition}\label{3.13} Let $R\subseteq S$ be an extension of rings and an upward directed  family $\{R_\alpha | \alpha \in I\}$ of elements of $[R,S]$ such that $R\subseteq R_\alpha$ is quasi-Pr\"ufer for each $\alpha \in I$. Then $R\subseteq \cup[R_\alpha | \alpha \in I]$ is quasi-Pr\"ufer.
\end{proposition}
\begin{proof} It is enough to use \cite[Proposition 5.13, p.54]{KZ} where $A_\alpha$  is the integral closure of $R$ in $R_\alpha$. 
\end{proof}

A ring morphism $R \to T$ preserves the integral closure of ring morphisms $R \to S$ if $\overline T^{T\otimes_RS} \cong  T \otimes_R \overline R $ for every ring morphism $R\to S$. An absolutely flat morphism $R\to T$ ($R\to T$ and $T\otimes_R T \to T$ are both flat) preserves integral closure \cite[Theorem 5.1]{O}. Flat epimorphisms, Henselizations and \'etale morphisms are absolutely flat.  Another examples are morphisms $R\to T$ that are essentially of finite type and (absolutely) reduced \cite[Proposition 5.19]{PR}(2). Such  morphisms are flat if $R$ is reduced \cite[Proposition 3.2]{LA}.

We will  prove an ascent result for absolutely flat ring morphisms. This will be proved by using base changes. For this we need to introduce some  concepts. A ring $A$ is called an AIC ring if each monic polynomial of $A[X]$ has a zero in $A$. We recalled  in \cite[p.4662]{PAIC} that any ring $A$ has a faithfully flat integral  extension $A\to A^*$, where $A^*$ is an AIC ring. Moreover, if $A$ is an AIC ring, each localization $A_P$ at a prime ideal $P$ of $A$ is a strict Henselian ring \cite[Lemma II.2]{PAIC}.

\begin{theorem}\label{3.14} Let $R\subseteq  S$ be a (quasi-) Pr\"ufer extension and $R\to T$ an absolutely flat ring morphism. Then $T \to T\otimes_R S$ is a (quasi-) Pr\"ufer extension.
\end{theorem}
\begin{proof}  

We can suppose that $R$ is an AIC ring. To see this, it is enough to use the base change $R\to R^*$. We set $T^*:= T\otimes_RR^*$, $S^*:= S\otimes_RR^*$. We first observe that 
$R^*\subseteq S^*$ is quasi-Pr\"ufer for the following reason: the composite extension $R\subseteq S \subseteq S^*$ is quasi-Pr\"ufer because the last extension is integral. Moreover, $R^* \to T^*$ is absolutely flat.  In case $T^*\subseteq T^*\otimes_{R^*} S^*$ is quasi-Pr\"ufer, so is $T\subseteq T\otimes_R S$,  because $T\to T^*= T\otimes_RR^*$ is faithfully flat and $T^*\subseteq T^*\otimes_{R^*} S^*$ is deduced from $T\subseteq _RS$ by the faithfully flat base change $T \to T\otimes_R S$. It is then enough to apply Proposition~\ref{3.17}.

We thus assume from now on that $R$ is an AIC ring.

 Let $N \in \mathrm{Spec}(T)$ be lying over $M$ in $R$. Then $R_M \to T_N$ is absolutely flat \cite[Proposition f]{O1} and $R_M \subseteq  S_M$ is quasi-Pr\"ufer. Now observe that $(T\otimes_R S)_N \cong T_N\otimes_{R_M} S_M$. Therefore, we can suppose that $R$ and $T$ are local and $R\to T$ is local and injective.  We deduce from \cite[Theorem 5.2]{O}, that $R_M \to T_N$ is an isomorphism. Therefore the proof is complete in the quasi-Pr\"ufer case. For the Pr\"ufer case, we need only to observe that absolutely flat morphisms preserve integral closure and a quasi-Pr\"ufer extension is Pr\"ufer if it is integrally closed. 
\end{proof}

\begin{proposition}\label{3.15}  Let $R \subseteq  S$ be an extension of rings and $R \to T$ a  base change which  preserves  integral closure. If $T \subseteq T\otimes_R S$ has FCP  and  $R \subseteq S$ is Pr\"ufer, then $T \subseteq  T\otimes_RS$ is Pr\"ufer.
\end{proposition}

\begin{proof} The result holds because   an FCP extension is Pr\"ufer if and only if it is integrally closed. 
\end{proof}

We observe that $T\otimes_R \widetilde {R} \subseteq  \widetilde{T}$ needs not to be an isomorphism, since this property may fail even for a localization $R \to R_P$, where $P$ is a prime ideal of $R$.

\begin{proposition}\label{3.16} Let $R\subseteq S$ be an extension of rings, $R\to R'$ a faithfully flat ring morphism and set $S' := R'\otimes_RS$.  If $R'\subseteq S'$ is (quasi-) Pr\"ufer (respectively, FCP), then so is $R\subseteq S$.
\end{proposition}
\begin{proof} The Pr\"ufer case is clear, because faithfully  flat morphisms descend flat epimorphisms (Scholium A (9)). For the quasi-Pr\"ufer case, we use the INC-pair characterization  and the fact that $\mathrm{F}_{R,S}(P)\to\mathrm{F}_{R',S'}(P')$ is faithfully flat  for $P'\in \mathrm{Spec}(R')$ lying over $P$ in $R$ \cite[Corollaire 3.4.9]{EGA}. The FCP case is proved in \cite[Theorem 2.2]{DPP3}. 
\end{proof}
\begin{proposition}\label{3.17} Let $R\subseteq S$ be a ring extension and $R\to R'$
a  spectrally surjective ring morphism  (for example, either faithfully flat or injective and integral). Then $R\subseteq S$ is quasi-Pr\"ufer if $R' \to R'\otimes_R S$ is injective (for example, if $R\to R'$ is  faithfully flat) and quasi-Pr\"ufer.
\end{proposition}
\begin{proof}
 Let $T\in [R,S]$ and  $P\in \mathrm{Spec}(R)$ and set $T':= T\otimes_RR'$. There is some $P'\in \mathrm{Spec}(R')$ lying over $P$, because $R\to R'$ is spectrally surjective. There is a faithfully flat  morphism $\mathrm{F}_{R,T}(P) \to   \mathrm{F}_{R',T'}(P') \cong  \mathrm{F}_{R,T}(P)\otimes_{\mathbf{k}(P)}\kappa(P')$ \cite[Corollaire 3.4.9]{EGA}.  By Theorem~\ref{2.3},  the result follows from  the faithful flatness of  $\mathrm{F}_{R,T}(P) \to   \mathrm{F}_{R',T\otimes_RR'}(P')$.
\end{proof}

\begin{theorem}\label{3.18}  Let $R\subseteq S$ be a ring extension. 
  
   \begin{enumerate} 
  \item   $R\subseteq S$ has a greatest quasi-Pr\"ufer subextension $R \subseteq \overset{\Longrightarrow} R =\widetilde{\overline R}$.
    
   \item $R \subseteq  \overline R \widetilde R= : \vec R$ is quasi-Pr\"ufer and then $\vec R \subseteq \overset{\Longrightarrow}R$.
   
   \item $\overline R^{\overset{\Longrightarrow}R} = \overline R$ and $\widetilde R^{\overset{\Longrightarrow}R} = \widetilde R$.
   \end{enumerate}
  \end{theorem}
  \begin{proof}
   To see (1), use Proposition~\ref{3.13} which tells us that the set of all quasi-Pr\"ufer subextensions is upward directed and then use Proposition~\ref{3.12} to prove the existence of $\overset{\Longrightarrow} R$. Then let $R\subseteq T \subseteq \overset{\Longrightarrow} R$ be a tower with $R\subseteq T$ integral and $T\subseteq \overset{\Longrightarrow} R$ Pr\"ufer. From $T\subseteq \overline R \subseteq \widetilde{\overline R} \subseteq \overset{\Longrightarrow} R$, we deduce that  $T= \overline R$ and then $ \overset{\Longrightarrow}  R = \widetilde{\overline R}$. 
   
  (2) Now $R \subseteq \overline R\widetilde R$ can be factored $R\subseteq \widetilde R \subseteq \overline R\widetilde R$ and is a tower of quasi-Pr\"ufer extensions, because $\widetilde R \to \widetilde R \overline R$ is integral.
  
  (3) Clearly, the integral closure and the Pr\"ufer closure of $R$ in $\overset{\Longrightarrow} R$ are the respective intersections of $\overline R$ and $\widetilde R$ with $\overset{\Longrightarrow}R$, and $\overline R,\widetilde R \subseteq \overset{\Longrightarrow}R$. 
 \end{proof}
 This last result means that, as long integral closures and Pr\"ufer closures of subsets of $\overset{\Longrightarrow}R$ are concerned, we can suppose that $R \subseteq S$ is quasi-Pr\"ufer.

\section{Almost-Pr\"ufer extensions}

We next give a definition ``dual" of the definition of a quasi-Pr\"ufer extension.
\subsection{Arbitrary extensions}

\begin{definition}\label{4.1} A ring extension $R\subseteq S$ is called an 
{\it almost-Pr\"ufer} extension if it can be factored $R\subseteq T \subseteq S$, where $R\subseteq T$ is Pr\"ufer and $T\subseteq S$ is integral. 
\end{definition}

\begin{proposition}\label{4.2} An extension $R\subseteq S$ is almost-Pr\"ufer if and only if $\widetilde R \subseteq S$ is integral. It follows that the subring $T$ of the above definition is $\widetilde R= \widehat R$ when $R\subseteq S$ is almost-Pr\"ufer.
\end{proposition}
\begin{proof} If $R\subseteq S$ is almost-Pr\"ufer, there is a factorization $R\subseteq T \subseteq \widetilde R\subseteq \widehat R \subseteq S$, where $T\subseteq \widehat R$ is both integral and a flat epimorphism by Scholium A (4). Therefore,   $T =\widetilde R= \widehat R$ by Scholium A (5) (L).
\end{proof}

\begin{corollary}\label{4.2bis} Let $R\subseteq S$ be a quasi-Pr\"ufer extension, and let $T\in[R,S]$. Then, $T\cap \overline R\subseteq T\overline R$ is almost-Pr\"ufer. Moreover, $T=\widetilde{\overline R\cap T}^{T\overline R}$.
\end{corollary}
\begin{proof} $T\cap \overline R\subseteq T$ is quasi-Pr\"ufer by Corollary ~\ref{3.3}.  Being integrally closed, it is Pr\"ufer by Corollary ~\ref{3.5}. Moreover, $T\subseteq T \overline R$ is an integral extension. Then, $T\cap \overline R\subseteq T\overline R$ is almost-Pr\"ufer and $T=\widetilde{\overline R\cap T}^{T\overline R}$.
\end{proof}

We note that  integral extensions and Pr\"ufer extensions are almost-Pr\"ufer and hence minimal extensions are almost-Pr\"ufer. There are quasi-Pr\"ufer extensions that are not almost-Pr\"ufer. It is enough to consider  \cite[Example 3.5(1)]{P2}. Let $R\subseteq T \subseteq S$ be two minimal extensions, where $R$ is local, $R\subseteq T$ integral  and $T\subseteq S$  is Pr\"ufer. Then $R\subseteq S$ is quasi-Pr\"ufer but not almost-Pr\"ufer, because $S = \widehat R$ and $R= \widetilde R$.  The same example shows that a composite of almost-Pr\"ufer extensions may not be almost-Pr\"ufer.

But the reverse implication holds.

\begin{theorem}\label{4.3} Let $R\subseteq S$ be an almost-Pr\"ufer extension. Then $R\subseteq S$ is quasi-Pr\"ufer. Moreover,   $\widetilde R = \widehat R$, $(\widetilde R)_P = \widetilde{R_P}$ for each $P \in \mathrm{Spec}(R)$. In this case, any flat epimorphic subextension $R\subseteq T$ is Pr\"ufer.
\end{theorem}
\begin{proof}  Let $R\subseteq \widetilde  R \subseteq S$, be an almost-Pr\"ufer extension, that is $\widetilde R \subseteq S$ is integral. 
The result follows because   $R\subseteq \widetilde R$ is Pr\"ufer.  Now the Morita hull and the Pr\"ufer hull coincide by Proposition~\ref{4.2}. 
In the same way,  $(\widetilde R)_P \to \widetilde{R_P}$ is a flat epimorphism and $(\widetilde R)_P \to S_P$ is  integral.
\end{proof}

We could define almost-Pr\"ufer rings as the rings $R$ such that $R\subseteq \mathrm{Tot}(R)$ is almost-Pr\"ufer. But in that case $\widetilde R = \mathrm{Tot}(R)$ (by Theorem~\ref{4.3}), so that $R$ is a Pr\"ufer ring. The converse evidently holds. Therefore, this concept does not define something new.

We observed in \cite[Remark 2.9(c)]{DPP2} that there is an almost-Pr\"ufer  FMC extension $R \subseteq S \subseteq T$, where $R\subseteq S$ is a Pr\"ufer minimal extension and  $S\subseteq T$ is minimal and integral.  
But $R \subseteq T$ is not an FCP extension.

\begin{proposition}\label{4.4}  Let $R \subseteq S$ be an extension verifying the hypotheses:

\begin{enumerate}

\item[(i)] $R \subseteq S$ is  quasi-Pr\"ufer.

\item[(ii)] $R\subseteq S$  can be factored $R\subseteq T \subseteq S$, where $R \subseteq T$ is a flat epimorphism.  
\end{enumerate}

\begin{enumerate}
 \item  Then the following commutative diagram (D) is a pushout,  
\car R,{\overline R},T,{TÊ\overline R}

\noindent  $T\overline R \subseteq S$ is Pr\"ufer and $R \subseteq T\overline R $ is quasi-Pr\"ufer. Moreover, $\mathrm{F}_{R,\overline R}(P) \cong \mathrm{F}_{T,T\overline R}(Q)$ for each $Q \in \mathrm{Spec}(T)$ and $P:= Q \cap R$.

\item  If in addition $R\subseteq T$ is integrally closed,  $(D)$ is a pullback, $T\cap \overline R = R$,  $(R:\overline R )= (T: T\overline R)\cap R$ and $(T:T\overline R) = (R:\overline R)T$.
\end{enumerate}
\end{proposition}
\begin{proof} 
(1) Consider the injective composite map $ \overline R \to \overline R \otimes_R T \to T\overline R$. As $ \overline R \to \overline R \otimes_R T$ is a flat epimorphism, because deduced by a base change of $R \to T$, we get that the surjective map $ \overline R \otimes_R T \to T\overline R$ is an isomorphism  by Scholium A (3).
By fibers transitivity, we have $\mathrm{F}_{T,\overline RT}(Q) \cong\kappa(Q) \otimes_{\mathbf{k}(P)}\mathrm{F}_{R,\overline R}(P)$ \cite[Corollaire 3.4.9]{EGA}. As $\kappa(P) \to\kappa(Q)$ is an isomorphism by Scholium A, we get that $\mathrm{F}_{R,\overline R}(P) \cong \mathrm{F}_{T,\overline RT}(Q)$.

(2) As in \cite[Lemma 3.5]{AD},  $R = T\cap \overline R$. The first statement on the conductors has the same proof as in \cite[Lemma 3.5]{AD}. The second holds because $R\subseteq T$  is a flat epimorphism (see Scholium A (6)).  
\end{proof}

\begin{theorem}\label{4.5} Let $R\subset S$ be a quasi-Pr\"ufer  
 extension and the diagram (D'):
\car R,{\overline R},\widetilde R,{\widetilde RÊ\overline R}

\begin{enumerate}

\item (D') is a pushout and a pullback, such that $\overline R\cap \widetilde R = R$ and $(R:\overline R )= (\widetilde R: \widetilde RÊ\overline R)\cap R$ so that $(\widetilde R:\widetilde RÊ\overline R) = (R:\overline R)\widetilde R$. 

\item $R\subset S$ can be factored  $R\subseteq \widetilde R \overline R= \overline{\widetilde{R}}= \vec R \subseteq \overset{\Longrightarrow}R= \widetilde{\overline{R}} = S$, where the first extension is  almost-Pr\"ufer and the second  is Pr\"ufer.  

\item $R\subset S$ is almost-Pr\"ufer if and only if $S= \overline R \widetilde R\Leftrightarrow \widetilde{\overline{R}}=\overline{\widetilde{R}}$.

\item $R\subseteq  \widetilde R \overline R= \overline{\widetilde{R}}= \vec R$ is the greatest almost-Pr\"ufer subextension of $R\subseteq S$ and $\widetilde R  =\widetilde R ^{\vec R}$.

\item $\mathrm{Supp}(S/R)=\mathrm{Supp}(\widetilde R/R)\cup\mathrm{Supp}(\overline{R}/R)$ if $R\subseteq S$ is almost-Pr\"ufer.  ($\mathrm{Supp}$ can be replaced with $\mathrm{MSupp}$).
\end{enumerate} 
\end{theorem}
\begin{proof} To show (1), (2), in view of Theorem~\ref{3.18}, it is enough to apply Proposition~\ref{4.4} with $T= \widetilde R$ and $S=\overset{\Longrightarrow} R$, because $R \subseteq \widetilde R\overline R$ is almost-Pr\"ufer whence quasi-Pr\"ufer, keeping in mind that a Pr\"ufer extension is integrally closed, whereas an integral Pr\"ufer extension is trivial. Moreover,  $\overline{\widetilde{R}}=\overline{R}\widetilde{R}$ because $\overline{R}\widetilde{R}\subseteq\overline{\widetilde{R}}$ is both integral and integrally closed.

(3) is obvious.

(4) Now consider an almost-Pr\"ufer subextension $R \subseteq T \subseteq U$, where $R\subseteq T$ is Pr\"ufer and $T\subseteq U$ is integral. Applying (3), we see that $U= {\overline R}^U  {\widetilde R}^U\subseteq \overline R \widetilde R $ in view of Proposition~\ref{0.8}. 

(5)   Obviously, $\mathrm{Supp}(\widetilde R/R)\cup\mathrm{Supp}(\overline{R}/R)\subseteq \mathrm{Supp}(S/R)$. Conversely, let $M\in \mathrm{Spec}(R)$ be such that $R_M\neq S_M$, and  $R_M = (\widetilde R)_M= \overline{R}_M$.  Then  (3) entails that $S_M = (\overline R )_M(\widetilde R)_M = R_M$, which is absurd.
\end{proof}

\begin{corollary}\label{4.5bis} Let $R \subseteq S$ be an almost-Pr\"ufer extension. The following conditions are equivalent:
\begin{enumerate}

\item $\mathrm{Supp}(S/\overline{R})\cap\mathrm{Supp}(\overline{R}/R)=\emptyset$.

\item $\mathrm{Supp}(S/\widetilde R)\cap\mathrm{Supp}(\widetilde{R}/R)=\emptyset$.

\item $\mathrm{Supp}(\widetilde R/R)\cap\mathrm{Supp}(\overline{R}/R)=\emptyset$.
\end{enumerate} 
\end{corollary}
\begin{proof}
Since $R \subseteq S$ is almost-Pr\"ufer, we get $ (\widetilde R)_P = \widetilde{R_P}$ for each $P\in\mathrm{Spec}(R)$. Moreover, $\mathrm{Supp}(S/R)=\mathrm{Supp}(\widetilde R/R)\cup\mathrm{Supp}(\overline{R}/R)=\mathrm{Supp}(S/\overline{R})\cup\mathrm{Supp}(\overline{R}/R)=\mathrm{Supp}(S/\widetilde R)\cup\mathrm{Supp}(\widetilde R/R)$.

(1) $\Rightarrow$ (2): Assume that there exists  $P\in\mathrm{Supp}(S/\widetilde R)\cap\mathrm{Supp}(\widetilde{R}/R)$. Then, $(\widetilde R)_P\neq S_P,R_P$, so that $R_P \subset S_P$ is neither Pr\"ufer, nor integral. But, $P\in\mathrm{Supp}(S/R)=\mathrm{Supp}(S/\overline{R})\cup\mathrm{Supp}(\overline{R}/R)$. If $P\in \mathrm{Supp}(S/\overline{R})$, then $P\not\in\mathrm{Supp}(\overline{R}/R)$, so that $(\overline{R})_P=R_P$ and $R_P \subset S_P$ is  Pr\"ufer, a contradiction.  If $P\in \mathrm{Supp}(\overline{R}/R)$, then $P\not\in\mathrm{Supp}(S/\overline{R})$, so that $(\overline{R})_P=S_P$ and $R_P \subset S_P$ is  integral, a contradiction.

(2) $\Rightarrow$ (3): Assume that there exists  $P\in\mathrm{Supp}(\widetilde R/R)\cap\mathrm{Supp}(\overline{R}/R)$. Then, $R_P\neq(\widetilde R)_P,(\overline{R})_P$, so that $R_P \subset S_P$ is neither Pr\"ufer, nor integral. But, $P\in\mathrm{Supp}(S/R)=\mathrm{Supp}(S/\widetilde{R})\cup\mathrm{Supp}(\widetilde{R}/R)$. If $P\in \mathrm{Supp}(S/\widetilde{R})$, then $P\not\in\mathrm{Supp}(\widetilde{R}/R)$, so that $(\widetilde{R})_P=R_P$ and $R_P \subset S_P$ is  integral, a contradiction.  If $P\in \mathrm{Supp}(\widetilde{R}/R)$, then $P\not\in\mathrm{Supp}(S/\widetilde{R})$, so that $(\widetilde{R})_P=S_P$ and $R_P \subset S_P$ is  Pr\"ufer, a contradiction.

(3) $\Rightarrow$ (1): Assume that there exists  $P\in\mathrm{Supp}(S/\overline R)\cap\mathrm{Supp}(\overline{R}/R)$. Then, $(\overline R)_P\neq R_P,S_P$, so that $R_P \subset S_P$ is neither Pr\"ufer, nor integral. But, $P\in\mathrm{Supp}(S/R)=\mathrm{Supp}(\overline{R}/R)\cup\mathrm{Supp}(\widetilde{R}/R)$. If $P\in \mathrm{Supp}(\widetilde{R}/R)$, then $P\not\in\mathrm{Supp}(\overline{R}/R)$, so that $(\overline{R})_P=R_P$ and $R_P \subset S_P$ is  Pr\"ufer, a contradiction.  If $P\in \mathrm{Supp}(\overline{R}/R)$, then $P\not\in\mathrm{Supp}(\widetilde{R}/R)$, so that $(\widetilde{R})_P=R_P$ and $R_P \subset S_P$ is  integral, a contradiction.
\end{proof}

Proposition~\ref{4.4} has the following similar statement proved by Ayache and Dobbs. It reduces to Theorem~\ref{4.5} in case $R\subseteq S$ has FCP because of Proposition~\ref{0.5}.

\begin{proposition}\label{4.6} Let $R\subseteq T \subseteq S$ be a quasi-Pr\"ufer extension,  where $T\subseteq S$ is an integral minimal  extension and $R \subseteq T$ is  integrally closed . Then the diagram (D) is  a pullback, $S= T\overline R$ and  $(T:S)= (R:\overline R)T$.
\end{proposition}
\begin{proof} \cite[Lemma 3.5]{AD}.
\end{proof}

 \begin{proposition}~\label{4.7} Let $R \subseteq U\subseteq S$ and $R\subseteq V \subseteq S$ be two towers of extensions, such that $R\subseteq U$ and $R\subseteq V$ are almost-Pr\"ufer. Then $R \subseteq UV$ is almost-Pr\"ufer and  $\widetilde{UV} =\widetilde{U}\widetilde{V}$.
\end{proposition} 
\begin{proof} Denote by $U'$, $V'$  and $W'$ the Pr\"ufer hulls of $R$ in $U$, $V$ and  $W=UV$. We deduce from \cite[Corollary 5.11, p.53]{KZ}, that $R\subseteq U'V' $ is Pr\"ufer. Moreover, $U'V'\subseteq UV$ is clearly  integral  and $U'V' \subseteq W'$ because the Pr\"ufer hull is the greatest Pr\"ufer subextension. We deduce that $R\subseteq UV$ is almost-Pr\"ufer and that $\widetilde{UV} =\widetilde{U}\widetilde{V}$.
\end{proof}

 \begin{proposition}~\label{4.8} Let $R \subseteq U\subseteq S$ and $R\subseteq V \subseteq S$ be two towers of extensions, such that $R\subseteq U$ is almost-Pr\"ufer and $R\subseteq V$  is a flat epimorphism. Then $U\subseteq UV$ is almost-Pr\"ufer.
\end{proposition}

\begin{proof} Mimic the proof of Proposition~\ref{4.7} and use \cite[Theorem 5.10, p.53]{KZ}.
\end{proof}
 
 \begin{proposition}~\label{4.9} Let  $R \subseteq S$ be an almost-Pr\"ufer extension and $R\to T $ a flat epimorphism. Then $T\subseteq T\otimes_R S$ is almost-Pr\"ufer.
\end{proposition}
  
\begin{proof} It is enough to use Proposition~\ref{3.10} and Definition~\ref{4.1}.
\end{proof}
   
\begin{proposition}\label{4.10} An extension $R\subseteq S$  is  almost-Pr\"ufer if and only if   $R_P \subseteq S_P$ is almost-Pr\"ufer and $\widetilde{R_P} = (\widetilde R)_P$ for each $P \in \mathrm{Spec}(R)$.    
  \end{proposition}
 \begin{proof}
   For an  arbitrary extension $R\subseteq S$ we have  $(\widetilde R)_P \subseteq \widetilde{R_P}$. Suppose that $R \subseteq S$ is almost-Pr\"ufer, then so is $R_P \subseteq S_P$  and $(\widetilde R)_P = \widetilde{R_P}$ by Theorem~\ref{4.3}. Conversely, if $R\subseteq S$ is locally almost-Pr\"ufer, whence locally quasi-Pr\"ufer, then $R\subseteq S$ is quasi-Pr\"ufer. If  $\widetilde{R_P} = (\widetilde R)_P$ holds for each $P \in \mathrm{Spec}(R)$, we have $S_P = (\overline R \widetilde R)_P$  so that $S = \overline R \widetilde R$ and $R \subseteq S$ is almost-Pr\"ufer  by  Theorem~\ref{4.5}.  
   \end{proof}
   
   \begin{corollary}\label{4.11} An FCP extension $R\subseteq S$ is almost-Pr\"ufer if and only if $R_P\subseteq S_P$ is almost-Pr\"ufer for each $P\in \mathrm{Spec}(R)$.
 \end{corollary}
 \begin{proof} It is enough to show that  $R\subseteq S$ is almost-Pr\"ufer if   $R_P\subseteq S_P$ is almost-Pr\"ufer for each $P\in \mathrm{Spec}(R)$ using Proposition~\ref{4.10}. Any minimal extension $\widetilde R\subset R_1$ is integral by definition of $\widetilde R$. Assume that $(\widetilde R)_P\subset \widetilde {(R_P)}$, so that there exists $R'_2\in[\widetilde R,S]$ such that $(\widetilde R)_P\subset(R'_2)_P$ is  a Pr\"ufer minimal extension with crucial maximal ideal $Q(\widetilde R)_P$, for some $Q\in\mathrm{Max}(\widetilde R)$ with $Q\cap R\subseteq P$. In particular, $\widetilde R\subset R'_2$ is not  integral. We may assume that there exists $R'_1\in[\widetilde R,R'_2]$ such that $R'_1\subset R'_2$ is a Pr\"ufer minimal extension with $P\not\in\mathrm{Supp}(R'_1/\widetilde R)$. Using \cite[Lemma 1.10]{P2}, there exists $R_2\in[\widetilde R,R'_2]$ such that $\widetilde R\subset R_2$ is a Pr\"ufer minimal extension with crucial maximal ideal $Q$, a contradiction. Then, $(\widetilde R)_P\subset S_P$ is integral for each $P$, whence $(\widetilde R)_P= \widetilde {(R_P)}$.
\end{proof}

  We now intend to demonstrate that our methods allow us to prove easily some results. For instance, next statement generalizes \cite[Corollary 4.5]{AD} and can be fruitful in algebraic number theory. 
  
  \begin{proposition}\label{4.12.1} Let $(R,M)$ be a one-dimensional local ring and $R\subseteq S$ a quasi-Pr\"ufer extension. Suppose that there is a tower $R\subset T \subseteq S$, where $R\subset T$ is integrally closed. Then $R\subseteq S$ is almost-Pr\"ufer, $T = \widetilde R$ and $S$ is zero-dimensional.
  \end{proposition}
  \begin{proof} Because $R\subset T$ is quasi-Pr\"ufer and integrally closed, it is Pr\"ufer. If some prime ideal of $T$ is lying over $M$, $R\subset T$ is a faithfully flat epimorphism, whence an isomorphism by Scholium A, which  is absurd. Now let $N$ be a prime ideal of $T$ and $P:= N\cap R$. Then $R_P$ is zero-dimensional and isomorphic to $T_N$. Therefore, $T$ is zero-dimensional. It follows that $T\overline R$ is zero-dimensional. Since $R\overline T \subseteq S$ is Pr\"ufer, we deduce from Scholium A, that  $\overline RT = S$. The proof is now complete.  
   \end{proof}
   We also generalize \cite[Proposition 5.2]{AD} as follows.
   
   \begin{proposition}\label{4.12.2} Let $R\subset S$ be a quasi-Pr\"ufer extension, such that $\overline R$ is local with maximal ideal $N:=\sqrt {(R:\overline R)}$. Then $R$ is local and $[R,S]= [R,\overline R]\cup [\overline R ,S]$. If in addition $R$ is one-dimensional, then either $R\subset S$ is integral or there is some minimal prime ideal $P$ of $\overline R$, such that $S= (\overline R)_P$, $P=SP$ and  $\overline R/P$ is a one-dimensional valuation domain    with  quotient field $S/P$.
  \end{proposition}
 
 \begin{proof} $R$ is obviously local. Let $T\in[R,S]\setminus[R,\overline R]$ and $s\in T\setminus \overline R$. Then $s\in$ U$(S)$ and $s^{-1}\in  \overline R$ by Proposition ~\ref{0.3} (1). But $s^{-1}\not\in$ U$(  \overline R)$, so that $s^{-1}\in N$. It follows that there exists some integer $n$ such that $s^{-n}\in (R:\overline R)$, giving $s^{-n}\overline R\subseteq R$, or, equivalently, $\overline R\subseteq Rs^n\subseteq T$. Then, $T\in[\overline R, S]$ and we obtain $[R,S]= [R,\overline R]\cup [\overline R ,S]$.
 
 Assume that $R$ is one-dimensional. 
 If $R\subset S$ is not integral then $\overline R \subset S$ is Pr\"ufer and $\overline R$ is one-dimensional. To complete the proof, use Proposition~\ref{0.3} (3).
  \end{proof}  
 
\subsection{FCP extensions}

In case we consider only FCP extensions, we obtain more results.

\begin{proposition}\label{4.12} Let $R\subseteq S$ be an FCP extension. The following statements are equivalent:
\begin{enumerate}
\item  $R\subseteq S$ is almost-Pr\"ufer.

\item   $R_P \subseteq S_P$ is either integral or Pr\"ufer   for each $P \in \mathrm{Spec}(R)$.

\item $R_P \subseteq S_P$ is almost-Pr\"ufer and $ \mathrm{Supp}(S/\widetilde R)\cap\mathrm{Supp}(\widetilde R/R)= \emptyset$.

\item $\mathrm{Supp}(\overline R/R)\cap \mathrm{Supp}(S/\overline R) = \emptyset  $. 

\end{enumerate}
\end{proposition}
\begin{proof}The equivalence of Proposition~\ref{4.10} shows that (2)  $\Leftrightarrow$ (1) holds because  $\widehat T =\widetilde T$  and over  a local ring $T$, an almost-Pr\"ufer  FCP extension $T\subseteq U$ is either integral or Pr\"ufer \cite[Proposition 2.4]{P2} . Moreover when $R_P\subseteq S_P$ is either integral or Pr\"ufer, it is easy to show that $(\widetilde R)_P = \widetilde{R_P}$

Next we show that (3) is equivalent to (2) of Proposition~\ref{4.10}.

Let $P\in\mathrm{Supp}(S/\widetilde R)\cap\mathrm{Supp}(\widetilde R/R)$ be such that $R_P\subseteq S_P$ is almost-Pr\"ufer. Then, $(\widetilde{R})_P\neq R_P,S_P$, so that  $R_P\subset (\widetilde{R})_P\subset S_P$. Since $R\subset\widetilde R$ is  Pr\"ufer, so is $R_P\subset (\widetilde{R})_P$, giving $(\widetilde R)_P\subseteq\widetilde{R_P}$ and $R_P\neq \widetilde{R_P}$. It follows that $\widetilde{R_P}=S_P$ in view of the dichotomy principle \cite[Proposition 3.3]{P2} since $R_P$ is a local ring, and then $\widetilde{R_P}\neq(\widetilde R)_P$. 

Conversely, assume that $\widetilde{R_P}\neq(\widetilde R)_P$, {\it i.e.} $P\in \mathrm{Supp}(S/R)$. Then, $R_P\neq\widetilde{R_P}$, so that $\widetilde{R_P}=S_P$, as we have just seen. Hence $R_P\subset S_P$ is integrally closed. It follows that $\overline{R_P}=\overline R_P=R_P$, so that $P\not\in\mathrm{Supp}(\overline R/R)$ and $P\in\mathrm{Supp}(\widetilde{R}/R)$ by Theorem~\ref{4.5}(5). Moreover, $\widetilde R_P\neq  S_P$ implies that $P\in\mathrm{Supp}(S/\widetilde R)$. To conclude, $P\in\mathrm{Supp}(S/\widetilde R)\cap\mathrm{Supp}(\widetilde R/R)$.

(1) $\Leftrightarrow$ (4)  An FCP extension is quasi-Pr\"ufer by Corollary~\ref{3.4}. Suppose that $R\subseteq S$ is almost-Pr\"ufer. By Theorem~\ref{4.5}, letting $U:= \widetilde R$, we get that $U\cap \overline R = R$ and $S=\overline R U$. We deduce from \cite[Proposition 3.6]{P2} that $\mathrm{Supp}(\overline R/R)\cap \mathrm{Supp}(S/\overline R) = \emptyset$. Suppose that this last condition holds. Then by \cite[Proposition 3.6]{P2} $R\subseteq S$ can be factored $R\subseteq U \subseteq S$, where $R\subseteq U$ is integrally closed, whence Pr\"ufer by Proposition~\ref{0.5}, and $U\subseteq S$ is integral. Therefore, $R\subseteq S$ is almost-Pr\"ufer.
\end{proof}

\begin{lemma}\label{4.12bis} Let $B\subset D$ and $C\subset D$ be two integral minimal  extensions and $A:=B\cap C$. If $A\subset D$ has FCP, then, $A\subset D$ is  integral. 
\end{lemma}

\begin{proof} Set $M:=(B:D)$ and  $N:=(C:D)$. 

If $M\neq N$, then, $A\subset D$ is  integral by \cite[Proposition 6.6]{DPPS}.

Assume that $M=N$. Then, $M\in\mathrm{Max}(A)$ by \cite[Proposition 5.7]{DPPS}. Let $B'$ be the integral closure of $A$ in $B$.   Then $M$ is also an ideal of $B'$, which is prime in $B'$, and then maximal in $B'$. If $A\subset D$ is an FCP extension, so is  $B'\subseteq B$, which is a flat epimorphism, and so is $B'/M\subseteq B/M$. Then, $B'=B$ since $B'/M$ is a field. It follows that $A\subseteq B$ is an  integral extension, and so is $A\subset D$.
\end{proof}

\begin{proposition}\label{4.12ter} Let $R\subset S$ be an FCP almost-Pr\"ufer extension. Then, $\widetilde R = \widehat R$ is the least $T\in[R,S]$ such that $T\subseteq S$  is  integral. 
\end{proposition}

\begin{proof} We may assume that $R\subset S$ is not integral. If there is some $U\in[R,\widetilde R]$ such that $U\subseteq \widetilde R$ is integral, then $U=\widetilde R$. Set $X:=\{T\in[R,S]\mid T\subseteq S$   integral$\}$. It follows that $\widetilde R$ is a minimal element of $X$. We are going to show that $\widetilde R$ is the least element of $X$.

Set $n:=\ell[\widetilde R,S]\geq 1$ and let  $ \widetilde R=R_0\subset R_1 \subset\cdots\subset R_{n-1}\subset R_n=S$ be a maximal chain of $ [\widetilde R,S]$, with length $n$. 
There does not exist a maximal chain of $ \widetilde R$-subalgebras of $S$ with length $>n$. Let $T\in X$. We intend to show that $T\in[\widetilde R,S]$. It is enough to choose $T$ such that $T$ is a minimal element of $X$. Consider the induction hypothesis: (H$_n$): $X\subseteq [\widetilde R,S]$ when $n:=\ell[\widetilde R,S]$.

We first show (H$_1$). If $n=1$, $\widetilde R\subset S$ is  minimal. Let $T\in X$ and $T_1\in[T,S]$ be such that $T_1\subset S$ is minimal. Assume that $T_1\neq \widetilde R$.  Lemma ~\ref{4.12bis} shows that $T_1\cap \widetilde R\subset \widetilde R$ is  integral, which contradicts  the beginning of the proof. Then, $T_1=\widetilde R$, so that $T=\widetilde R$ for the same contradiction and (H$_1$) is proved.

Assume that $n>1$ and that  (H$_k$) holds for any $k<n$.  Let $T\in X$ and $T_1\in[T,S]$ be such that $T_1\subset S$ is minimal. If $T_1\in[ \widetilde R,S]$, then $k:=\ell[\widetilde R,T_1]\leq n-1$. But we get that $T\in[R,T_1]$, with $T\subseteq T_1$ integral. Moreover, $\widetilde R$ is also the Pr\"ufer hull of $R\subseteq T_1$, with $k:=\ell[\widetilde R,T_1]\leq n-1$. Since (H$_k$) holds, we get that $T\in[\widetilde R,T_1]\subset [\widetilde R,S]$. 

 If $T_1\not\in[ \widetilde R,S]$, set $U:=T_1\cap R_{n-1}$. We get that $T_1\subset S$ and $R_{n-1}\subset S$ are minimal and  integral. Using again Lemma ~\ref{4.12bis}, we get that $U\subset S$ is  integral, with $\ell[\widetilde R,R_{n-1}]= n-1$ and $U\in[R,R_{n-1}]$. As before, $\widetilde R$ is also the Pr\"ufer hull of $R\subseteq R_{n-1}$. Since (H$_{n-1}$) holds, $U\in[\widetilde R,R_{n-1}]$, so that $T_1\in[ \widetilde R,S]$, a contradiction. Therefore,  (H$_n$) is proved.
\end{proof}

We will need a relative version of the support. Let  $f: R\to T$ be a ring morphism and $E$ a $T$-module. The relative support of $E$ over $R$ is $\mathscr S_R(E):= {}^af(\mathrm{Supp}_T(E))$ and $\mathrm{M}\mathscr{S}_R(E):=\mathscr S_R(E)\cap\mathrm{Max}(R)$. In particular, for a ring extension $R\subset S$, we have $\mathscr S_R(S/R):= \mathrm{Supp}_R(S/R))$.

\begin{proposition}\label{4.13} Let $R\subseteq S$ be an FCP extension. The following statements hold:

\begin{enumerate}

 \item $\mathrm {Supp}(\overline{\widetilde{R}}/\overline{R})\cap\mathrm{Supp}(\overline{R}/R)=\emptyset$.

 \item $\mathrm{Supp}(\widetilde{R}/R)\cap\mathrm{Supp}(\overline{R}/R)=\mathrm{Supp}(\overline{\widetilde{R}}/\widetilde{R})\cap\mathrm{Supp}(\widetilde{R}/R)= \emptyset$. 
 
\item $\mathrm{MSupp}(S/R)=\mathrm{MSupp}(\widetilde R/R)\cup\mathrm{MSupp}(\overline{R}/R)$.
\end{enumerate}
\end{proposition}
\begin{proof}

(1) is a consequence of Proposition~\ref{4.12}(4) because $R\subseteq \overline{\widetilde{R}}$ is almost-Pr\"ufer. 

We  prove the first part of (2). If some $M \in \mathrm{Supp}(\widetilde R/R) \cap \mathrm{Supp}(\overline R/R)$, it can be supposed in $\mathrm{Max}(R)$. Set $R':=R_M,U:=(\widetilde R)_M, T:=(\overline R)_M$ and $M':=MR_M$. Then, $R'\neq U,T$, with $R'\subset U$ FCP Pr\"ufer and $R'\subset T$ FCP integral, an absurdity \cite[Proposition 3.3]{P2}. 

To show the second part, assume that some $P \in \mathrm{Supp}(\overline{\widetilde{R}}/\widetilde{R})\cap\mathrm{Supp}(\widetilde{R}/R)$. Then, $P\not\in\mathrm{Supp}(\overline{R}/R)$ by the first part of (2), so that $\overline{R}_P=R_P$, giving $(\overline{\widetilde{R}})_P=\overline{R}_P\widetilde{R}_P=\widetilde{R}_P$, a contradiction. 
 
(3) Obviously, 
$\mathrm{MSupp}(S/R)=\mathrm{M}\mathscr{S}(S/R)=\mathrm{M}\mathscr{S}(S/\overline T^S)\cup\mathrm{M}\mathscr{S}(\overline T^S/T)$

\noindent $\cup\mathrm{M}\mathscr{S}(T/\overline U ^T)\cup\mathrm{M}\mathscr{S}(\overline{U}^T/U)\cup\mathrm{M}\mathscr{S}(U/R)$.
By \cite[Propositions 2.3 and 3.2]{P2}, we have $\mathrm{M}\mathscr{S}(S/\overline T^S)\subseteq\mathscr{S}(\overline T^S/T)=\mathscr{S}(\overline R/\overline R^T)=\mathrm{M}\mathscr{S}(\overline R/R)=\mathrm{MSupp}(\overline R/R),\ \mathrm{M}\mathscr{S}(T/\overline U^T)=\mathscr{S}(\overline R^T/R)\subseteq\mathscr{S}(\overline R/R)=\mathrm{Supp}(\overline R/R)$ and $\mathrm{M}\mathscr{S}(\overline U^T/U)=\mathscr{S}(\overline R^T/R)=\mathrm{Supp}(\overline R/R)$. 
To conclude, $\mathrm{MSupp}(S/R)=\mathrm{MSupp}(\widetilde R/R)\cup\mathrm{MSupp}(\overline R/R)$.
\end{proof}

\begin{proposition}\label{4.14} Let $R\subset S$ be an FCP  extension and  $M\in\mathrm{MSupp}(S/R)$,  then $\widetilde{R_M}=(\widetilde R)_M$ if and only if $M\not\in\mathrm{MSupp}(S/\widetilde R)\cap\mathrm{MSupp}(\widetilde R/R)$.
\end{proposition}

\begin{proof} In fact, we are going to show that $\widetilde{R_M}\neq(\widetilde R)_M$ if and only if $M\in\mathrm{MSupp}(S/\widetilde R)\cap\mathrm{MSupp}(\widetilde R/R)$.

Let $M\in\mathrm{MSupp}(S/\widetilde R)\cap\mathrm{MSupp}(\widetilde R/R)$. Then, $\widetilde{R_M}\neq R_M,S_M$ and then  $R_M\subset\widetilde{R_M}\subset S_M$. Since $R\subset\widetilde R$ is  Pr\"ufer, so is $R_M\subset\widetilde{R_M}$ by Proposition~\ref{0.3}, giving $(\widetilde R)_M\subseteq\widetilde{R_M}$ and $R_M\neq \widetilde{R_M}$. Therefore, $\widetilde{R_M}=S_M$ \cite[Proposition 3.3]{P2} since $R_M$ is  local, and then $\widetilde{R_M}\neq(\widetilde R)_M$. 

Conversely, if $\widetilde{R_M}\neq(\widetilde R)_M$, then, $R_M\neq\widetilde{R_M}$, so that $\widetilde{R_M}=S_M$, as we have just seen and then $R_M\subset S_M$ is integrally closed. It follows that $\overline{R_M}=\overline R_M=R_M$, so that $M\not\in\mathrm{MSupp}(\overline R/R)$. Hence, $M\in\mathrm{MSupp}(\widetilde{R}/R)$ by Proposition~\ref{4.13}(3). Moreover, $\widetilde R_M\neq  S_M\Rightarrow M\in\mathrm{MSupp}(S/\widetilde R)$. To conclude, $M\in\mathrm{MSupp}(S/\widetilde R)\cap\mathrm{MSupp}(\widetilde R/R)$.
\end{proof}

If $R\subseteq S$ is any ring extension, with $\dim(R) =0$, then $\widetilde{R_M}=(\widetilde R)_M$ for any $M\in\mathrm{Max}(R)$. Indeed by Scholium A (2), the flat epimorphism $R\to \widetilde R$ is bijective as well as $R_M\to (\widetilde R)_M$. This conclusion is still valid in another context.

\begin{corollary}\label{4.15} Let $R\subset S$ be an FCP  extension. Assume that one of the following conditions is satisfied:

\begin{enumerate}
\item  $\mathrm{MSupp}(S/\widetilde R)\cap\mathrm{MSupp}(\widetilde R/R)=\emptyset$.

\item $S=\overline R\widetilde R$, or equivalently, $R \subseteq S$ is almost-Pr\"ufer.
\end{enumerate} 

Then, $\widetilde{R_M}=(\widetilde R)_M$ for any $M\in\mathrm{Max}(R)$.
\end{corollary}

\begin{proof}
(1) is Proposition ~\ref{4.14}.
(2) is Proposition ~\ref{4.10}.
\end{proof}

\begin{proposition}\label{4.15bis} Let $R\subset S$ be an almost-Pr\"ufer FCP  extension.  Then, any $T\in[R,S]$ is the integral closure of $T\cap \widetilde R$ in $T\widetilde R$. 
\end{proposition}

\begin{proof} Set $U:=T\cap \widetilde R$ and $V:=T \widetilde R$. Since $R\subset S$ is almost-Pr\"ufer,  $U\subseteq \widetilde R$ is  Pr\"ufer and $\widetilde R\subseteq V$ is integral and  $\widetilde{R}$  is also the Pr\"ufer hull of $U\subseteq V$. Because $R\subset S$ is almost-Pr\"ufer,  for each $M\in\mathrm{MSupp}_R(S/R)$,  $R_M\subseteq S_M$ is either integral, or Pr\"ufer by Proposition ~\ref{4.12}, and so is $U_M\subseteq V_M$. But $\widetilde{R_M}=(\widetilde R)_M$ by Corollary ~\ref{4.15}  is also the Pr\"ufer hull of $U_M\subseteq V_M$. Let $T'$ be the integral closure of $U$ in $V$. Then, $T'_M$ is the integral closure of $U_M$ in $V_M$. 

Assume that  $U_M\subseteq V_M$ is integral. Then $V_M=T'_M$ and $U_M=(\widetilde R)_M$, so that $V_M=T_M (\widetilde R)_M=T_M$, giving $T_M=T'_M$.

Assume that  $U_M\subseteq V_M$ is Pr\"ufer. Then $U_M=T'_M$ and $V_M=(\widetilde R)_M$, so that $U_M=T_M \cap(\widetilde R)_M=T_M$, giving $T_M=T'_M$.

To conclude, we get that $T_M=T'_M$  for each $M\in\mathrm{MSupp}_R(S/R)$. Since $R_M=S_M$, with $T_M=T'_M$ for each $M\in\mathrm{Max}(R)\setminus\mathrm{MSupp}_R(S/R)$, we get $T=T'$,  whence  $ T$ is the  integral closure of  $U\subseteq V$.
\end{proof}

We build an example of an FCP extension $R\subset S$ where we have $\widetilde{R_M}\neq(\widetilde R)_M$ for some $M\in\mathrm{Max}(R)$. In particular, $R\subset S$ is not almost-Pr\"ufer.

\begin{example}\label{4.16} Let $R$ be an integral domain with quotient field $S$ and $\mathrm{Spec}(R):=\{M_1,M_2,P,0\}$, where $M_1\neq  M_2$ are two  maximal ideals and $P$ a prime ideal satisfying $P\subset M_1\cap M_2$. Assume that there are $R_1,\ R_2$ and $R_3$ such that $R\subset R_1$ is  Pr\"ufer minimal, with $\mathscr{C}(R,R_1)=M_1,\ R\subset R_2$ is  integral minimal, with $\mathscr{C}(R,R_2)=M_2$ and $R_2\subset R_3$ is Pr\"ufer minimal, with $\mathscr{C}(R_2,R_3)=M_3\in\mathrm{Max}(R_2)$ such that $M_3\cap R=M_2$ and $M_2R_3=R_3$. This last condition is satisfied when $R\subset R_2$ is either ramified or inert. Indeed, in both cases, $M_3R_3=R_3$; moreover, in the ramified case, we have $M_3^2\subseteq M_2$ and in the inert case, $M_3=M_2$ \cite[Theorem 3.3]{Pic}. We apply \cite[Proposition 7.10]{DPPS} and \cite[Lemma 2.4]{DPP2} several times. Set $R'_2:=R_1R_2$. Then, $R_1\subset R'_2$ is  integral minimal,  with $\mathscr{C}(R_1,R'_2) =:M'_2=M_2R_1$ and $R_2\subset R'_2$ is  Pr\"ufer minimal, with $\mathscr{C}(R_2,R'_2)=:M'_1=M_1R_2\in\mathrm{Max}(R_2)$. Moreover, $M'_1\neq M_3,\ \mathrm {Spec}(R_1)=\{M'_2,P_1,0\}$, where $P_1$ is the only prime ideal of $R_1$ lying over $P$. But, $P=(R:R _1)$ by \cite[Proposition 3.3]{FO}, so that $P=P_1$. Set $R'_3:=R_3R'_2$. Then, $R'_2\subset R'_3$ is  Pr\"ufer minimal, with $\mathscr{C}(R'_2,R'_3)=:M'_3=M_3R'_2\in\mathrm {Max}(R'_2)$ and $R_3\subset R'_3$ is  Pr\"ufer minimal, with $\mathscr {C}(R_3,R'_3)=M''_1=M_1R_3\in\mathrm{Max}(R_3)$. It follows that we have $\mathrm{Spec}(R'_3)=\{P',0\}$ where $P'$ is the only prime ideal of $R'_3$ lying over $P$. To end, assume that $R'_3\subset S$ is  Pr\"ufer minimal, with $\mathscr{C}(R'_3,S)=P'$. Hence, $R_2$ is the integral closure of $R$ in $S$.  In particular, $R\subset S$ has FCP \cite[Theorems 6.3 and  3.13]{DPP2} and is quasi-Pr\"ufer. Since $R\subset R_1$ is integrally closed, we have $R_1\subseteq\widetilde R$. Assume that $R_1\neq\widetilde R$. Then, there exists $T\in[R_1,S]$ such that $R_1\subset T$ is  Pr\"ufer minimal and $\mathscr{C}(R_1,T)=M'_2$, a contradiction by Proposition~\ref{4.12} since $M'_2=\mathscr{C}(R_1,R'_2)$, with $R_1\subset R'_2$ integral  minimal. Then, $R_1=\widetilde R$. It follows that $M_1\in \mathrm{MSupp}(\widetilde R/R)$. But, $P=\mathscr{C}(R'_3,S)\cap R\in\mathrm{Supp}(S/\widetilde R)$ and $P\subset M_1$ give $M_1\in\mathrm{MSupp}(S/\widetilde R)$, so that $\widetilde{R_{M_1}}\neq(\widetilde R)_{M_1}$ by Proposition~\ref{4.14} giving that $R\subset S$ is not almost-Pr\"ufer.

\end{example}

We now intend to refine Theorem~\ref{4.5}, following the scheme used in \cite[Proposition 4]{A} for extensions of integral domains.

\begin{proposition}\label{4.17} Let $R\subseteq S$ and $U, T\in [R,S]$ be such that $R\subseteq U$ is integral and $R\subseteq T$ is Pr\"ufer.  Then $U\subseteq UT$ is Pr\"ufer in the following cases and $R\subseteq UT$ is almost-Pr\"ufer.
\begin{enumerate}
\item  $\mathrm{Supp}(\overline R/R) \cap \mathrm{Supp}(\widetilde R/R) = \emptyset$ (for example, if $R\subseteq S$ has FCP).

\item  $R\subseteq  U$ preserves integral closure.
\end{enumerate}
\end{proposition}
\begin{proof}  (1) We have $\emptyset =\mathrm{MSupp}(U/R) \cap \mathrm{MSupp}(T/R)$, since $U\subseteq \overline R$ and $T\subseteq \widetilde R$. 
Let $M \in \mathrm{MSupp}((UT)/R)$. For $M\in\mathrm{MSupp}(U/R)$, we have $R_M=T_M$ and $(UT)_M = U_M$.
If $M\notin \mathrm{MSupp}(U/R)$, then $U_M= R_M$ and $(UT)_M= T_M$, so that $U_M \subseteq (UT)_M$ identifies to $R_M \subseteq T_M$. 

Let $N\in\mathrm{Max}(U)$ and set $M:=N\cap R\in\mathrm{Max}(R)$ since $R\subseteq U$ is integral. If $M\not\in\mathrm{Supp}(\overline R/R)$, then $R_M=\overline R_M=U_M$ and $N$ is the only maximal ideal of $U$ lying over $M$. It follows that $U_M=U_N$ and $(UT)_M=(UT)_N$ by \cite[Lemma 2.4]{DPP2}. Then,  
$U_N \subseteq (UT)_N$ identifies to $R_M \subseteq T_M$ which is Pr\"ufer. If $M\not\in\mathrm{Supp}(\widetilde R/R)$, then $R_M=T_M$ gives $U_M=(UT)_M$, so that $U_N=(UT)_N$ by localizing the precedent equality and 
$U_N \subseteq (UT)_N$ is still Pr\"ufer. Therefore,  $U\subseteq UT$ is locally Pr\"ufer, whence Pr\"ufer by 
Proposition~\ref{0.2}. 

(2) The usual reasoning shows that $U\otimes_R T\cong UT$, so that  $U\subseteq UT$ is integrally closed. Since $U $ is contained in $\overline R^{UT}$, we get that $U=\overline R^{UT}$. Now observe that $R\subseteq UT$ is almost-Pr\"ufer, whence quasi-Pr\"ufer. It follows that $U\subseteq UT$ is Pr\"ufer. 
\end{proof}

Next propositions generalize  Ayache's results of \cite[Proposition 11]{A}.

\begin{proposition}\label{4.18} Let $R\subseteq S$ be a quasi-Pr\"ufer extension, $T, T'\in [R,S]$ and  $U:=T\cap T'$.  The following statements hold:

\begin{enumerate} 
\item $\widetilde T = \widetilde{(T\cap \overline R)}$ for each $T\in [R,S]$.

\item $\widetilde T\cap \widetilde{T'} \subseteq \widetilde{T\cap T'}$.

\item Let $\mathrm{Supp}(\overline T/T) \cap \mathrm{Supp}(\widetilde T/T)=  \emptyset$ (this assumption holds if $R\subseteq S$ has FCP). Then,  $T\subseteq T' \Rightarrow \widetilde T \subseteq \widetilde{T'}$.

\item  If $\mathrm{Supp}(\overline U/U) \cap \mathrm{Supp}(\widetilde U/U)=  \emptyset$, then $\widetilde T\cap \widetilde{T'} = \widetilde{T\cap T'}$.
\end{enumerate}

\end{proposition}

\begin{proof}(1) We observe that $R\subseteq T$ is quasi-Pr\"ufer  by Corollary~\ref{3.3}. Since $T\cap \overline R$ is the integral closure of $R$ in $T$, we get that $T\cap \overline R \subseteq T$ is Pr\"ufer.
It follows that $T\cap \overline R \subseteq \widetilde T$ is Pr\"ufer.  We thus have $\widetilde T \subseteq \widetilde{T\cap \overline R}$. To prove the reverse inclusion, we set $V:= T\cap \overline R$ and $W:= \widetilde V\cap \overline T$. We have $W\cap \overline R= \widetilde V \cap \overline R =  V$, because $V \subseteq \widetilde V \cap \overline R   $ is integral and Pr\"ufer since we have a tower $V \subseteq \widetilde V\cap \overline R \subseteq  \widetilde V$. Therefore, $V\subseteq W$ is  Pr\"ufer because $W\in[V,\widetilde V]$. Moreover, $T\subseteq  \widetilde T\subseteq  \widetilde V$, since $V\subseteq  \widetilde T$ is Pr\"ufer. Then,  $T\subseteq W$  is integral because $W\in[T,\overline T]$, and we have  $V \subseteq T \subseteq W$. This entails that $T= W= \widetilde V\cap  \overline T$, so that  $T\subseteq \widetilde V$ is  Pr\"ufer. It follows that $\widetilde V \subseteq \widetilde T$ since $T\in[V,\widetilde V]$.

(2) A quasi-Pr\"ufer extension is Pr\"ufer if and only if it is integrally closed.  We observe that $T\cap T' \subseteq \widetilde T \cap \widetilde{T'}$ is integrally closed, whence Pr\"ufer. It follows that $\widetilde T\cap \widetilde{T'}\subseteq \widetilde{T\cap T'}$. 

(3) Set $U = T\cap \overline R$ and $U' = T'\cap \overline R$, so that $U, U' \in [R,\overline R]$ with $U\subseteq U'$. In view of (1), we thus can suppose that $T, T'\in [R,\overline R]$. It follows that $T\subseteq T'$ is integral and $T\subseteq \widetilde  T$ is Pr\"ufer. We deduce from Proposition~\ref{4.17}(1) that $T'\subseteq T'\widetilde T$ is Pr\"ufer, so that $\widetilde{T} T'\subseteq \widetilde{T'}$, because $\mathrm{Supp}(\overline T/T) \cap \mathrm{Supp}(\widetilde T/T)=  \emptyset$ and $\overline T=\overline R$. Therefore, we have $\widetilde T \subseteq \widetilde{T'}$.

(4) Assume that $\mathrm{Supp}(\overline U/U) \cap \mathrm{Supp}(\widetilde U/U)=  \emptyset$. Then, $T\cap T'\subset T,T'$ gives $\widetilde{T\cap T'}\subseteq \widetilde{T} \cap \widetilde{T'}$  in view of (3), so that $\widetilde{T\cap T'}= \widetilde{T} \cap \widetilde{T'}$ by (2).
\end{proof}

\begin{proposition}\label{4.19} Let $R\subseteq S$ be a quasi-Pr\"ufer extension and $T\subseteq T'$ a subextension of $R\subseteq S$.
 Set $U:=T\cap \overline R,\ U':= T'\cap \overline R,\  V:=T \overline R$ and $V':=T' \overline R$. The following statements hold:

\begin{enumerate} 
\item $ T \subseteq T'$ is integral if and only if $V=V'$.

\item $ T \subseteq T'$ is Pr\"ufer if and only if $U=U'$.

\item Assume that $U\subset U'$ is integral  minimal and $V=V'$. Then, $T\subset T'$ is integral  minimal, of the same type as $U\subset U'$. 

\item Assume that $V\subset V'$ is  Pr\"ufer minimal  and $U=U'$. Then, $T\subset T'$ is  Pr\"ufer minimal.

\item Assume that $T\subset T'$ is  minimal    and set $P:=\mathcal{ C}(T,T')$.

\noindent (a) If $T\subset T'$ is  integral, then $U\subset U'$ is  integral minimal  if and only if $P\cap U\in\mathrm{Max}(U)$.

\noindent(b) If $T\subset T'$ is Pr\"ufer, then $V\subset V'$ is Pr\"ufer minimal  if and only if there is exactly one prime ideal in $V$ lying over $P$.\end{enumerate}

\end{proposition}

\begin{proof} In $[R,S]$ we have the integral extensions $U\subseteq U',\ T\subseteq V,\ T'\subseteq V'$ and the Pr\"ufer extensions $V\subseteq V',\ U\subseteq T,\ U'\subseteq T'$. Moreover, $ \overline R$ is also the integral closure of $U\subseteq V'$. 

(1) is gotten by considering the extension $T\subseteq V'$, which is both $T\subseteq V\subseteq V'$ and $T\subseteq T'\subseteq V'$.

(2) is gotten by considering the extension $U\subseteq T'$, which is both $U\subseteq T\subseteq T'$ and $U\subseteq U'\subseteq T'$.

(3) Assume that $U\subset U'$ is integral  minimal  and $V=V'$. Then, $T\subset T'$ is  integral  by (1) and $T\neq T'$ because of (2). Set $M:=(U:U')\in\mathrm{Supp}_U(U'/U)$. For any $M'\in\mathrm{Max}(U)$ such that $M'\neq M$, we have $U_{M'}=U'_{M'}$, so that $T_{M'}={T'}_{M'}$ because $U_{M'}\subseteq T'_{M'}$ is Pr\"ufer. But, $U\subseteq T'$ is almost-Pr\"ufer, giving $T'=TU'$. By Theorem ~\ref{4.5}, $(T:T')=(U:U')T=MT\neq T$ because $T\neq T'$. We get that $U\subseteq T$ Pr\"ufer implies that $M\not\in \mathrm{Supp}_U(T/U)$ and $U_M=T_M$.   It follows that ${T'}_M=T_M{U'}_M={U'}_M$. Therefore, $T_M\subseteq {T'}_M$ identifies to $U_M\subseteq {U'}_M$, which is minimal  of the same type as $U\subset U'$ by \cite[Proposition 4.6]{DPPS}. Then, $T\subset T'$ is integral minimal, of the same type as $U\subset U'$. 

(4) Assume that $V\subset V'$ is  Pr\"ufer minimal and $U=U'$. Then, $T\subset T'$ is Pr\"ufer  by (2) and $T\neq T'$ because of (1). Set $Q:=\mathcal{ C}(V,V')$ and $P:=Q\cap T\in\mathrm{Max}(T)$ since $Q\in\mathrm{Max}(V)$. For any $P'\in\mathrm{Max}(T)$ such that $P'\neq P$, and $Q'\in\mathrm{Max}(V)$ lying above $P'$, we have $V_{Q'}=V'_{Q'}$, so that $V_{P'}={V'}_{P'}$. It follows that ${T'}_{P'}\subseteq {V'}_{P'}$ is  integral, so that $T_{P'}={T'}_{P'}$ and $P'\not\in \mathrm{Supp}_T(T'/T)$. We get that $T\subset T'$ is Pr\"ufer minimal in view of \cite[Proposition 6.12]{DPP2}. 

(5) Assume that $T\subset T'$ is a minimal  extension and set $P:=\mathcal{ C}(T,T')$.

(a) Assume that  $T\subset T'$ is  integral. Then, $V=V'$ and $U\neq U'$ by (1) and (2). We can use Proposition ~\ref{4.4} getting that $P=(U:U')T\in\mathrm{Max}(T)$ and $Q:=(U:U')=P\cap U\in\mathrm{Spec}(U)$. It follows that $Q\not\in\mathrm{Supp}_U(T/U)$, so that $U_Q=T_Q$ and $U'_Q=T'_Q$. Then,  $U_Q\subset U'_Q$ is  integral minimal, with $Q\in\mathrm{Supp}_U(U'/U)$. 

If $Q\not\in\mathrm{Max}(U)$, then $U\subset U'$ is not  minimal   by the properties of the crucial maximal ideal. 

Assume that $Q\in\mathrm{Max}(U)$ and let $M\in\mathrm{Max}(U)$, with $M\neq Q$. Then, $U_M=U'_M$ because  $M+Q=U$, so that $U\subset U'$ is a minimal extension and  (a) is gotten. 

(b) Assume that  $T\subset T'$ is Pr\"ufer. Then, $V\neq V'$ and $U= U'$ by (1) and (2). Moreover, $PT'=T'$ gives $PV'=V'$.  Let $Q\in\mathrm{Max}(V)$ lying over $P$. Then, $QV'=V'$ gives that $Q\in
\mathrm{Supp}_V(V'/V)$. Moreover, we have $V'=VT'$. Let $P'\in \mathrm{Max}(T),\ P'\neq P$.  Then, $T_{P'}=T'_{P'}$ gives $V_{P'}=V'_{P'}$. It follows that $\mathrm{Supp}_T(V'/V)=\{P\}$ and $\mathrm{Supp}_V(V'/V)=\{Q\in\mathrm{Max}(V)\mid Q\cap T=P\}$. But, by \cite[Proposition 6.12]{DPP2},  $V\subset V'$ is Pr\"ufer minimal   if and only if $|\mathrm{Supp}_V(V'/V)|=1$, and then  if and only if there is exactly one prime ideal in $V$ lying over $P$. 
\end{proof}

\begin{lemma}\label{4.20} Let $R\subseteq S$ be an FCP almost-Pr\"ufer extension and $U\in[R,\overline R]$, $V\in[\overline R,S]$. Then  $U\subseteq V$ has FCP and is  almost-Pr\"ufer.  
\end{lemma}

\begin{proof} Obviously, $U\subseteq V$ has FCP and $\overline R$ is  the integral closure of $U$ in $V$.  Proposition ~\ref{4.12} entails that  $\mathrm{Supp}_R(\overline R/R)\cap \mathrm{Supp}_R(S/\overline R) = \emptyset  $. We claim that  $\mathrm{Supp}_U(\overline R/U)\cap \mathrm{Supp}_U(V/\overline R) =\emptyset  $. Deny  and let $Q\in\mathrm{Supp}_U(\overline R/U)\cap \mathrm{Supp}_U(V/\overline R)$. Then, $\overline R_Q\neq U_Q,V_Q$. If $P:=Q\cap R$. we get that $\overline R_P\neq U_P,V_P$, giving $\overline R_P\neq R_P,S_P$, a contradiction. Another use of Proposition ~\ref{4.12} shows that $U\subseteq V$ is   almost-Pr\"ufer.  
\end{proof}

\begin{proposition}\label{4.21} Let $R\subseteq S$ be an FCP almost-Pr\"ufer extension and $T\subseteq T'$ a subextension of $R\subseteq S$.
 Set $U:=T\cap \overline R$ and $V':=T' \overline R$. Let $W$ be the Pr\"ufer hull of $U\subseteq V'$. Then, $W$ is also the Pr\"ufer hull of $T\subseteq T'$ and $T\subseteq T'$ is an FCP almost-Pr\"ufer extension.
 \end{proposition}

\begin{proof} By Lemma ~\ref{4.20}, we get that $U\subseteq V'$ is an FCP almost-Pr\"ufer extension.  Let  $\widetilde T$ be the Pr\"ufer hull of $T\subseteq T'$. Since $U\subseteq T $ and $T\subseteq \widetilde T$  are Pr\"ufer, so is $U\subseteq \widetilde T$ and $\widetilde T\subseteq V'$  gives that $\widetilde T\subseteq W$. Then,  $T\subseteq W$ is Pr\"ufer as a subextension of $U\subseteq W$. 

Moreover, in view of Proposition ~\ref{4.12ter}, $W$ is the least $U$-subalgebra of $V'$ over which $V'$ is integral. Since $T'\subseteq V'$ is integral, we get that $W\subseteq T'$, so that $W\in[T,T']$, with $W\subseteq T'$ integral as a subextension of $W\subseteq V'$. It follows that $W$ is also the Pr\"ufer hull of $T\subseteq T'$ and $T\subseteq T'$ is an FCP almost-Pr\"ufer extension.
\end{proof}

\section{The case of Nagata extensions}

In this section we transfer the quasi-Pr\"ufer (and almost-Pr\"ufer) properties to Nagata extensions.
\begin{proposition}\label{0.4} Let $R\subseteq S$ be a Pr\"ufer (and FCP) extension, then $R(X) \subseteq S(X)$ is a Pr\"ufer (and FCP) extension.
   \end{proposition}
   \begin{proof} We can suppose that $(R,M)$ is local, in order to use Proposition~\ref{0.3}(3). Then it is  enough to know the following facts: $V(X)$ is a valuation domain if so is  $V$; $R[X]_{P[X]} \cong R(X)_{P(X)} \cong R_P(X)$ where $P(X) =PR(X)$  and $ R(X)/P(X)\cong (R/P)(X) $ for $P\in \mathrm{Spec}(R)$. If in addition $R\subseteq S$ is FCP, it is enough to use  \cite[Theorem 3.9]{DPP3}: $R\subset S$ has FCP  if and only if $R(X)\subset S(X)$ has FCP.
    \end{proof}

 \begin{proposition}\label{nagataqp} If $R\subseteq S$ is quasi-Pr\"ufer, then so is $R(X)\subseteq S(X)$, $\overline{R(X)}=\overline R(X)\cong \overline R \otimes_R R(X)$ and $S(X)\cong S\otimes_R R(X)$.
 \end{proposition}
 \begin{proof} It is enough to use proposition~\ref{0.4}, because $\overline{R(X)}
 = \overline{R}(X)$. The third assertion results from \cite[Proposition 4 and Proposition 7]{PicC}.
 \end{proof}

\begin{proposition}\label{4.6} If $R \subseteq S$ is almost-Pr\"ufer, then so is $R(X)\subseteq S(X)$. It follows that   $\widetilde{R(X)} = \widetilde{R}(X)$ for an  almost-Pr\"ufer extension $R\subseteq S$.
\end{proposition}
\begin{proof} 
If $R \subseteq S$ is almost-Pr\"ufer, then $R\subseteq \widetilde{R}$ is Pr\"ufer and $\widetilde{R}\subseteq S$ is integral and then $R(X)\subseteq \widetilde{R}(X)$ is Pr\"ufer and $\widetilde{R}(X)\subseteq S(X)$ is integral, whence $R(X)\subseteq S(X)$ is almost-Pr\"ufer with $\widetilde{R(X)} = \widetilde{R}(X)$.
\end{proof}

\begin{lemma}\label{5.4} Let $R\subset S$ be an FCP ring extension such that $\widetilde{R}=R$. Then, $\widetilde{R (X)}=R(X)$.
\end{lemma}

\begin{proof} If $R(X)\neq\widetilde{R (X)}$, there is some $T'\in[R(X),\widetilde{R (X)}]$ such that $R(X)\subset T'$ is  Pr\"ufer minimal. Set $ \mathscr{C}(R(X),T')\in\mathrm{MSupp}(S(X)/R(X))$ $=: M'$. There is $M\in\mathrm{MSupp}(S/R)$ such that $M'=MR(X)$  \cite[Lemma 3.3]{DPP3}. But, $M'\not\in\mathrm{MSupp}(\overline{R(X)}/R(X))=\mathrm{MSupp}(\overline{R}(X)/R(X))$ by Proposition~\ref{4.13}(2), giving that $M\not\in\mathrm {MSupp}(\overline{R}/R)=\mathscr{S}(\overline{R}/R)$. Then \cite[Proposition 1.7(3)]{P2} entails that $M\in\mathscr{S}(S/\overline{R})$. By \cite[Proposition 1.7(4)]{P2}, there are some $T_1,T_2\in[\overline{R},S]$  with $T_1\subset T_2$   Pr\"ufer minimal (an FCP extension is quasi-Pr\"ufer), with $M=\mathscr{C}(T_1,T_2)\cap R$. We can  choose  for $T_1\subset T_2$  the first minimal extension verifying the preceding  property. Therefore, $M\not\in\mathscr{S}(T_1/\overline{R})$, so that $M\not\in\mathscr{S}(T_1/R)=\mathrm{Supp}(T_ 1/R)$. By \cite[Lemma 1.10]{P2}, we get that there exists $T\in[R,T_2]$ such that $R\subset T$ is   Pr\"ufer minimal, a contradiction. 
\end{proof}

\begin{proposition}\label{5.5} If $R\subset S$ is an FCP  extension, then, $\widetilde{R}(X)=\widetilde{R (X)}$.
\end{proposition}

\begin{proof}  
Because $R\subseteq\widetilde{R}$ is Pr\"ufer, $R(X)\subseteq\widetilde{R}(X)$ is Pr\"ufer by Corollary~\ref{0.4}. Then, $\widetilde{R}(X)\subseteq\widetilde{R (X)}$. Assume that $\widetilde{R}(X)\neq\widetilde{R(X)}$ and set $T:=\widetilde{R}$, so that $T=\widetilde{T}$, giving $\widetilde{T(X)}=T(X)=\widetilde{R}(X)$ by Lemma~\ref{5.4}. Hence $\widetilde{T(X)}\subset\widetilde{R(X)}$ is a Pr\"ufer extension,  contradicting  the definition of $\widetilde{T (X)}$. So, $\widetilde{R}(X)=\widetilde{R (X)}$.
\end{proof}

\begin{proposition}\label{5.7} Let $R \subseteq S$ be an almost-Pr\"ufer FCP extension, then $\widehat{R}(X) = \widehat{R(X)} =\widetilde{R(X)}$.
\end{proposition}

\begin{proof} We have a tower $R(X) \subseteq \widehat{R(X)} = \widetilde{R(X)}=\widetilde{R}(X)  = \widehat R(X)$, where the first and the third equalities come from Theorem ~\ref{4.3} and the second from Proposition ~\ref{5.5}.
\end{proof}

 We end this section  with a special result.

\begin{proposition} \label{5.8} Let $R \subseteq S$ be an extension such that $R(X) \subseteq S(X)$ has FIP, then $\widehat{R}(X) = \widehat{R(X)}$.

\end{proposition}

\begin{proof} The map $[R,S] \to [R(X),S(X)]$ defined by $T \mapsto T(X) = R(X)\otimes_R T$ is bijective  \cite[Theorem 32]{DPP4},  whence $ \widehat{R(X)} = T(X)$ for some $T \in [R,S]$. Moreover, $\widehat{R}(X) \to \widehat{R(X)}$ is a flat epimorphism. Since $R \to R(X)$ is faithfully flat,  $\widehat R = T$
and the result follows.
\end{proof}

\section{Fibers of quasi-Pr\"ufer extensions}

We intend to complete some results of Ayache-Dobbs \cite{AD}.
We begin by recalling some features about quasi-finite ring morphisms. A ring morphism $R \to S$ is called quasi-finite  by \cite{R} if it is of finite type and $\kappa(P) \to\kappa(P)\otimes_RS$ is finite (as a $\kappa(P)$-vector space), for each $P\in\mathrm{Spec}(R)$ \cite[Proposition 3, p.40]{R}. 

\begin{proposition}\label{0} A ring morphism of finite type  is incomparable if and only if  it is quasi-finite and, if and only if its  fibers are finite. 
\end{proposition}
\begin{proof}
Use \cite[Corollary 1.8]{U}  and the above definition.
\end{proof}

\begin{theorem} \label{6.0} An extension  $R\subseteq S$  is quasi-Pr\"ufer if and only if  $R\subseteq T$ is quasi-finite (respectively, has finite fibers) for each $T\in [R,S]$ such that $T$ is of finite type over $R$, if and only if $R\subseteq T$ has integral fiber morphisms for each $T\in [R,S]$.
\end{theorem}
\begin{proof} It is clear that $R\subseteq  S$ is an INC-pair  implies  the condition because of Proposition~\ref{0}. To prove the converse, let $T\in [R,S]$ and write $T$ as the union of its finite type $R$-subalgebras $T_\alpha$. Now let $Q \subseteq Q'$ be prime ideals of $T$, lying over a prime ideal $P$ of $R$ and set $Q_\alpha : = Q \cap T_\alpha$ and $Q'_\alpha:= Q'\cap T_\alpha$. If $R \subseteq T_\alpha$ is quasi-finite, then $Q_\alpha = Q'_\alpha$, so  that $Q= Q'$ and then $R\subseteq T$ is incomparable. The last statement is Proposition~\ref{3.8}.
\end{proof}
\begin{corollary} An integrally closed extension is Pr\"ufer if and only if each of its subextensions $R\subseteq T$ of finite type has finite fibers.
\end{corollary}
\begin{proof} It is enough to observe that the fibers of a (flat) epimorphism  have a cardinal $\leq 1$, because an epimorphism is spectrally injective.
\end{proof}

A ring extension $R\subseteq S$ is called {\it strongly affine} if each of its subextensions $ R\subseteq T$ is of finite type. The above considerations show that in this case $R\subseteq S$ is quasi-Pr\"ufer if and only if each of its subextensions  $R\subseteq T$ has finite fibers. For example, an FCP extension is strongly affine and quasi-Pr\"ufer. We also are interested in extensions $R\subseteq S$ that are not necessarily strongly affine and such that each of its subextensions $R \subseteq T$ have finite fibers.

Next lemma will be useful, its proof is obvious.

\begin{lemma}\label{6.5.1}  Let $R \subseteq S$ be an extension and $T\in [R,S]$
\begin{enumerate}
\item If $T\subseteq S$ is spectrally injective and $R\subseteq T$ has finite fibers, then $R\subseteq S$ has finite fibers.
\item If $R\subseteq T$ is spectrally injective, then $T\subseteq S$ has finite fibers if and only if  $R\subseteq S$ has finite fibers.
\end{enumerate}
\end{lemma}

\begin{remark}\label{6.2} Let  $R\subseteq S$ be an almost-Pr\"ufer  extension, such that the integral extension $T:= \widetilde R \subseteq S$ has finite fibers and  let $P\in \mathrm{Spec}(R)$. The study of the finiteness of $\mathrm{Fib}_{R,S} (P)$ can be reduced as follows. 
As $\overline R \subseteq S$ is an epimorphism, because it is Pr\"ufer, it is spectrally injective (see Scholium A). The hypotheses of Proposition~\ref{4.4}  hold.  We examine three cases.  In case $(R:\overline R)\not\subseteq P$, it is well known that $R_P =(\overline R)_P$ so that $|\mathrm{Fib}_{R,S}(P)|= 1$, because  $\overline R \to S$ is  spectrally injective. Suppose now that $(R:\overline R) = P$.  From $(R:\overline R )= (T: S)\cap R$, we deduce that $P$ is lain over by some  $Q\in \mathrm{Spec}(T)$  and then $\mathrm{Fib}_{R,\overline R}(P) \cong \mathrm{Fib}_{T,S}(Q)$. The conclusion follows as above.
Thus the remaining case is $(R:\overline R ) \subset P$ and we can assume that $PT=T$ for if not $\mathrm{Fib}_{R,\overline R}(P) \cong \mathrm{Fib}_{T,S}(Q)$ for some  $Q \in \mathrm{Spec}(T)$ by Scholium A (1).  
\end{remark}

  \begin{proposition}\label{6.3}  Let $R\subseteq S$ be an almost-P\"rufer extension.  If  $\widetilde R \subseteq S$ has finite fiber morphisms and $(\widetilde R_P:S_P)$ is a maximal ideal of $\widetilde R_P$ for each $P \in \mathrm{Supp}_R(S/\widetilde R)$, then $R\subseteq \overline R$ and $R\subseteq S$ have finite fibers.
\end{proposition} 
\begin{proof} The Pr\"ufer closure commutes with the localization at prime ideals by  Proposition~\ref{4.10}. We set $T:= \widetilde R$. Let $P$ be a prime ideal of $R$ and $\varphi : R \to R_P$ the canonical morphism. We clearly have $\mathrm{Fib}_{R,.}(P) = {}^a\varphi(\mathrm{Fib}_{R_P,._P}(PR_P))$. Therefore, we can localize the data at $P$ and we can assume that $R$ is local.  
  
 In case $(T:S) = T$, we get a factorization  $R \to \overline R \to T$. Since $R \to T$ is Pr\"ufer so is $R \to \overline R$ and it follows that $R = \overline R$ because a Pr\"ufer extension is integrally closed. 

From Proposition ~\ref{0.3} applied to $R \subseteq T$,  we get that there is some $\mathfrak P \in \mathrm{Spec}(R)$ such that  $T= R_{\mathfrak{P}}$, $R/\mathfrak P$ is  a valuation ring with quotient field $T/\mathfrak{P}$ and $\mathfrak{P} =\mathfrak{P}T$. It follows that $(T:S)  = \mathfrak{P}T = \mathfrak{P} \subseteq R$, and hence $(T:S) = (T:S)\cap R =  (R:\overline R)$. We have therefore  a pushout diagram by Theorem~\ref{4.5}:
 
 \car {R':=R/\mathfrak{P}}, { \overline R/\mathfrak{P}:= \overline{R'}}, {T':= T/\mathfrak{P}}, {S/\mathfrak{P}:= S'}
 
\noindent where $R/\mathfrak{P}$ is a valuation domain, $T/\mathfrak{P}$ is  its quotient field and $\overline R/\mathfrak{P} \to S/\mathfrak{P}$ is Pr\"ufer by \cite[Proposition 5.8, p. 52]{KZ}.
 
 Because $\overline{R'} \to S'$ is injective and a flat epimorphism, there is a bijective map $\mathrm{Min}(S') \to \mathrm{Min}(\overline{R'})$. But $T' \to S'$ is the fiber at $\mathfrak{P}$ of $T \to S$ and is therefore finite. Therefore,   $\mathrm{Min}(S')$ is a finite set $\{N_1,\ldots, N_n\}$ of maximal ideals lying over the minimal prime ideals $\{M_1,\ldots, M_n\}$ of $\overline{R'}$ lying over $0$ in $R'$. We infer from Lemma~\ref{3.7} that $\overline{R'}/{M_i} \to S'/{N_i}$ is Pr\"ufer, whence integrally closed. Therefore, $\overline{R'}/{M_i}$ is  an integral domain and the integral closure of $R'$ in $S'/{N_i}$. Any maximal ideal $M$ of $\overline{R'}$ contains some $M_i$. To conclude it is enough to use a result of Gilmer \cite[Corollary 20.3]{MIT} because the number of maximal ideals in $\overline{R'}/M_i$ is less than the separable degree of the extension of fields $T' \subseteq S'/N_i$.
\end{proof}

\begin{remark}\label{6.4} (1) Suppose that $(\widetilde R:S)$ is a maximal ideal of $\widetilde R$. We clearly have $(\widetilde R:S)_P \subseteq (\widetilde R_P:S_P)$ and  the hypotheses on $(\widetilde R:S)$ of the above proposition hold. 

(2)  In case $\widetilde R\subseteq S$ is a tower of finitely many integral minimal extensions $R_{i-1} \subseteq R_i$  with $M_i= (R_{i-1}:R_i)$, then $\mathrm{Supp}_{\widetilde R}(S/\widetilde R) =\{N_1,\ldots,N_n\}\subseteq \mathrm{Max}(\widetilde R)$  where $N_i= M_i\cap R$. If the  ideals $N_i$ are different, each localization at $N_i$  of $\widetilde R\subseteq S$ is  integral  minimal and the above result may apply. 
This generalizes the Ayache-Dobbs result \cite[Lemma 3.6]{AD}, where $\widetilde R \subseteq S$ is supposed to be integral   minimal. 
\end{remark}

\begin{proposition}\label{6.5} Let $R \subseteq S$ be a quasi-Pr\"ufer ring extension. 

\begin{enumerate}

\item  $R \subseteq S$ has finite fibers if and only if $R\subseteq \overline R$ has finite fibers.

\item   $R\subseteq \overline R$ has finite fibers if and only if each extension $R \subseteq T$, where $T \in [R,S]$ has finite fibers.
\end{enumerate}

\end{proposition}

\begin{proof}  (1) Let $P\in \mathrm{Spec}(R)$ and the  morphisms $\kappa(P) \to\kappa(P)\otimes_R\overline R \to\kappa(P)\otimes_R S$. The first (second) morphism is integral (a flat epimorphism) because  deduced by base change from the  integral morphism $R \to \overline R$ (the flat epimorphism $\overline R \to S$). Therefore,  the ring $\kappa(P)\otimes_R \overline R$ is zero dimensional, so that the second morphism is surjective by Scholium A (2).  
Set $A: =\kappa(P)\otimes_R \overline R$ and $B:= \kappa(P)\otimes_R S$, we thus have a  module finite flat ring morphism $A \to B$. Hence, $A_Q \to B_Q$ is free for each $Q \in \mathrm{Spec}(A)$ \cite[Proposition 9]{E} and $B_Q \neq 0$ because it contains $\kappa(P) \neq 0$. Therefore, $A_Q \to B_Q$ is injective and it follows that $A\cong B$.

(2) Suppose that $R \subseteq \overline R$ has finite fibers and let $T \in [R,S]$, then $\overline R \subseteq \overline R T$ is a flat epimorphism by Proposition~\ref{4.4}(1) and so is $\kappa(P)\otimes_R \overline R \to\kappa(P)\otimes_R \overline RT$. Since $\mathrm{Spec}(\kappa(P)\otimes_R \overline RT) \to \mathrm{Spec}(\kappa(P)\otimes_R \overline R )$ is injective, $R \subseteq \overline RT$ has finite fibers. Now $R \subseteq T$ has finite fibers because $T \subseteq \overline RT$ is integral and is therefore spectrally surjective.
\end{proof}

\begin{remark}\label{6.6} Actually, the statement (1) is valid if we only suppose that  $\overline R \subseteq S$ is a flat epimorphism.
\end{remark}

Next result contains  \cite[Lemma 3.6]{AD}, gotten after a long proof.

\begin{corollary} Let $R\subseteq S$ be an almost-Pr\"ufer extension. Then $R\subseteq S$ has finite fibers if and only if $R\subseteq \overline R$ has finite fibers, and if and only if $\widetilde R \subseteq S$ has finite fibers.
\end{corollary}
\begin{proof} By Proposition~\ref{6.5}(1) the first equivalence is clear. The second is a consequence of Lemma~\ref{6.5.1}(2).
\end{proof}

The following result is then clear.

\begin{theorem} \label{6.7}Let $R\subseteq S$ be a quasi-Pr\"ufer extension with finite fibers, then $R\subseteq T$ has finite fibers for each $T \in [R,S]$. 
\end{theorem}

\begin{corollary}\label{6.8}  If $R\subseteq S$ is quasi-finite and quasi-Pr\"ufer, then $R\subseteq T$ has finite fibers for each $T \in [R,S]$ and $\widetilde R \subseteq S$ is module finite.
\end{corollary}
\begin{proof} By the Zariski Main Theorem, there is a factorization $R \subseteq F \subseteq  S$ where $R \subseteq F$ is module finite and $F \subseteq S$ is a flat epimorphism \cite[Corollaire 2, p.42]{R}.  To conclude, we use Scholium A in the rest of the proof. The map  $ \widetilde R\otimes_R F \to S $ is injective because $F \to \widetilde R\otimes_R F$ is a flat  epimorphism and is surjective, since it is integral and a flat epimorphism because $\widetilde R\otimes_R F \to S$ is a flat epimorphism .
\end{proof}

\begin{corollary}\label{6.9} An FMC extension $R \subseteq S$ is such that $R\subseteq T$ has finite fibers for each $T\in [R,S]$. 
\end{corollary}
\begin{proof}
Such an extension is quasi-finite and quasi-Pr\"ufer. Then use Corollary~\ref{6.8}.
\end{proof}

\cite[Example 4.7]{AD} exhibits  some FMC extension $R\subseteq S$, such that $R\subseteq \overline R$ has not FCP. Actually, $[R,\overline R]$ is an infinite (maximal) chain.

\begin{proposition}\label{6.10} Let $R \subseteq S$ be a quasi-Pr\"ufer extension such that $R \subseteq \overline R$ has finite fibers and $R$ is semi-local.
Then  $T$ is semi-local for each $T\in [R,S]$.
\end{proposition}

\begin{proof} Obviously $\overline R$ is semi-local. From the tower $\overline R \subseteq T\overline R \subseteq S$ we deduce that $\overline R \subseteq T \overline R$ is Pr\"ufer. It follows that $T\overline R$ is semi-local   \cite[Lemma 2.5 (f)]{AD}. As $T \subseteq T\overline R $ is integral, we get that $T$ is semi-local.
\end{proof}

The following proposition gives a kind of converse.

\begin{proposition}\label{6.15} Let $R \subseteq S$ be an extension with $ \overline R$  semi-local. 
Then $R \subseteq S$ is quasi-Pr\"ufer if and only if  $T$ is semi-local for each $T\in [R,S]$.
\end{proposition}

\begin{proof} If $R \subseteq S$ is quasi-Pr\"ufer,  $\overline R \subseteq S$ is Pr\"ufer. Let $T\in [R,S]$ and set $T':=T\overline R$, so that $T \subseteq T'$ is integral, and $\overline R \subseteq T'$ is Pr\"ufer (and then a normal pair). It follows from \cite[Lemma 2.5 (f)]{AD} that $T'$ is semi-local, and so is $T$.

If  $T$ is semi-local for each $T\in [R,S]$,  so is any $T\in [\overline R,S]$. Then, $(\overline R,S)$ is a residually algebraic pair \cite[Theorem 3.10]{AJ} (generalized to arbitrary extensions) and so is $\overline R_M\subseteq S_M$ for each $M\in\mathrm{Max}(\overline R)$, whence is Pr\"ufer  \cite[Theorem 2.5]{AJ} (same remark) and Proposition ~\ref{0.3}. Then, $\overline R \subseteq S$ is Pr\"ufer by Proposition ~\ref{0.2} and $R \subseteq S$ is quasi-Pr\"ufer. 
\end{proof}

\section{Numerical properties of FCP extensions}

\begin{lemma}\label{7.1} Let $R\subset S$ be an FCP extension. The map $\varphi:[R,S]\to\{(T',T'')\in[R,\overline R]\times[\overline R,S]\mid\mathrm{Supp}_{T'}(\overline R/T')\cap\mathrm{Supp}_{T '}(T''/\overline R)=\emptyset\}$, defined by $\varphi(T):=(T\cap\overline R,\overline RT)$ for each $T\in[R,S]$, is bijective. In particular, if $R\subset S$ has FIP, then $|[R,S]|\leq|[R,\overline R]||[\overline R,S]|$. 
\end{lemma}

\begin{proof} Let $(T',T'')\in[R,\overline R]\times[\overline R,S]$. Then, $\overline R$ is also the integral closure  of $T'$ in $T ''$ (and in $S$). 

Let $T\in[R,S]$. Set $T':=T\cap\overline R$ and $T'':=\overline RT$. Then $(T',T'')\in[R,\overline R]\times [\overline R,S]$. Assume that $T'=T''$, so that $T'=T''=\overline R$, giving $T=\overline R$ and $\mathrm {Supp}_{T'}(\overline R/T')=\mathrm{Supp}_{T'}(T''/\overline R)=\emptyset$. Assume that $T'\neq T''$. In view of \cite[Proposition 3.6]{P2}, we get $\mathrm{Supp}_{T'}(\overline R/T')\cap\mathrm{Supp}_{T'}(T''/\overline R)=\emptyset$. Hence  $\varphi$ is well defined.

Now, let $T_1,T_2\in[R,S]$ be such that $\varphi(T_1)=\varphi(T_2)=(T',T'')$. Assume $T'\neq T''$. Another use of \cite[Proposition 3.6]{P2} gives that $T_1=T_2$. If $T'=T''$, then, $T'=T''=\overline R$, so that $T_1=T_2=\overline R$. It follows that $\varphi$ is injective. The same reference gives that $\varphi$ is bijective. 
\end{proof}
 
\begin{proposition}\label{7.2} Let $R\subset S$ be a FCP extension. 
We define two order-isomorphisms   $\varphi'$ and $\psi$ as follows:

\item $\varphi':[R,\vec R]\to[R,\overline{R}]\times[\overline{R},\vec R]$ defined by $\varphi'(T):=(T\cap\overline{R},T\overline{R})$ 

\item $\psi:[R,\vec R]\to [R,\widetilde{R}]\times[\widetilde{R},\vec R]$ defined by $\psi(T):=(T\cap\widetilde{R},T\widetilde {R})$. 
\end{proposition}

\begin{proof} This follows from \cite[Lemma 3.7]{P2} and Proposition~\ref{4.13}. (We recall that $\vec R = \overline{\widetilde{R}}$.)
\end{proof}

\begin{corollary}\label{7.3} If   $R\subseteq S$ has FCP, then
 $\mathrm{Supp}(\vec R/\widetilde{R})=\mathrm{Supp}(\overline{R}/R)$, $\mathrm {Supp}(\vec R/\overline{R})=\mathrm{Supp}(\widetilde{R}/R)$ and  $\mathrm{Supp}(\vec R/R)=\mathrm{Supp}(\widetilde{R}/R)\cup \mathrm{Supp}(\overline{R}/R)$.
\end{corollary}

\begin{proof} Set $A:=\mathrm{Supp}(\overline{\widetilde{R}}/\widetilde{R}),\ B:=\mathrm{Supp}(\widetilde {R}/R),\ C:=\mathrm{Supp}(\overline{\widetilde{R}}/\overline{R})$ and $D:=\mathrm{Supp}(\overline{R}/R)$. Then, $A\cup B=C\cup D=\mathrm{Supp}(\overline{\widetilde{R}}/R)$, with $A\cap B=C\cap D=B\cap D=\emptyset$ by Proposition~\ref{4.13}. 

Assume that $A\cup B\neq B\cup D$ and let $P\in(A\cup B)\setminus(B\cup D)$. Then, $R_P\neq (\overline{\widetilde{R}})_P=(\overline{R})_P(\widetilde{R})_P=R_P$, a contradiction. It follows that $A\cup B=B\cup D$. Intersecting the two members of this equality with $A$ and $D$, we get $A=A\cap D=D$. In the same way, intersecting the equality $A\cup B=C\cup D=C\cup A$ by $B$ and $C$, we get $B=C$. 
\end{proof}

\begin{corollary}\label{7.4} Let $R\subset S$ be an FCP extension. We define two order-isomorphisms

\item $\varphi_1:[R,\widetilde{R}]\to[\overline{R},\vec R]$ by $\varphi_1(T):=T\overline {R}$

\item $\psi_1:[R,\overline{R}]\to[\widetilde{R},\vec R]$ by $\psi_1(T):=T\widetilde {R}$.
\end{corollary}

\begin{proof} We use notation of Proposition~\ref{4.13}. We begin to remark that $\overline{R}$ and $\widetilde{R}$ play symmetric roles.

Let $T,T'\in[R,\widetilde{R}]$ be such that $\varphi_1(T)=\varphi_1(T')$. Since $T\cap\overline{R}=T'\cap\overline{R}=R$ by Proposition~\ref{4.13}, we get $\varphi(T)=\varphi(T')$, so that $T=T'$ and $\varphi_1$ is injective. A similar argument shows that $\psi_1$ is injective.

Let $U\in[\overline{R},\overline{\widetilde{R}}]$. There exists $T\in[R,\overline{\widetilde{R}}]$ such that $\varphi(T)=(R,U)$, so that $R=T\cap\overline{R}$ and $U=T\overline{R}$. Let $M\in\mathrm{Supp}(\overline{\widetilde{R}}/R)=\mathrm{Supp}(\widetilde{R}/R)\cup\mathrm{Supp}(\overline{R}/R)$ by Corollary~\ref{7.3}. If $M\in\mathrm{Supp}(\widetilde{R}/R)$, then $M\not\in\mathrm{Supp}(\overline{R}/R)$ by Proposition~\ref{4.13}, giving $T_M\subseteq\overline{\widetilde{R}}_M=\overline{R}_M\widetilde {R}_M=\widetilde{R}_M$. If $M\in\mathrm{Supp}(\overline{R}/R)$, the same reasoning gives $T_M\subseteq\overline{R}_M$, so that $R_M=T_M\cap\overline{R}_M=T_M$, but $R_M=\widetilde{R}_M$. Then, $T_M=\widetilde{R}_M$. It follows that $T\subseteq\widetilde{R}$, giving $T\in[R,\widetilde{R}]$ and $\varphi_1$ is surjective, hence bijective. A similar argument shows that $\psi_1$ is surjective, hence bijective.
\end{proof} 

\begin{corollary}\label{7.5} If $R\subset S$ has FCP, then $\theta:[R,\widetilde{R}]\times[R,\overline{R}]\to[R,\
\vec R]$  defined by $\theta(T,T'):=TT'$,  is an order-isomorphism.
In particular, if $R\subset S$ has FIP, then $|[R,\widetilde{R}]||[R,\overline{R}]= |[R,\vec R] \|\leq |[R,S]|$.
\end{corollary}

\begin{proof} Using notation of Proposition~\ref{7.2} and Corollary~\ref{7.4}, we may remark that $\psi\circ \theta={\rm Id}\times \psi _1$. Since $\psi$ and  ${\rm Id}\times \psi _1$ are order-isomorphisms, so is $\theta$. The FIP case is obvious. 
\end{proof} 

Gathering the previous results, we get the following theorem.

\begin{theorem}\label{7.6} If $R\subset S$ has  FCP,  the next statements are equivalent:
\begin{enumerate}
\item $\mathrm{Supp}(\overline{R}/R)\cap\mathrm{Supp}(S/\overline{R})=\emptyset$.

\item The map $\varphi:[R,S]\to[R,\overline{R}]\times[\overline{R},S]$ defined by $\varphi(T):=(T\cap\overline{R},T\overline{R})$ is an order-isomorphism.

\item  $R\subseteq S$ is almost-Pr\"ufer.

\item $\mathrm{Supp}(S/\overline{R})=\mathrm{Supp}(\widetilde{R}/R)$.

\item The map $\varphi_1:[R,\widetilde{R}]\to[\overline{R},S]$ defined by $\varphi_1(T):=T\overline {R}$ is an order-isomorphism.

\item The map  $\psi_1:[R,\overline{R}]\to[\widetilde{R},S]$ defined by $\psi_1(T):=T\widetilde {R}$ is an order-isomorphism.

\item The map  $\theta:[R,\widetilde{R}]\times[R,\overline{R}]\to[R,S]$ defined by $\theta(T,T'):=TT'$ is an order-isomorphism.
\end{enumerate}

\noindent If one of these conditions holds, then $\mathrm{Supp}(S/\widetilde{R})=\mathrm{Supp}(\overline{R}/R)$.

If $R\subset S$ has FIP, the former conditions are equivalent to each of the following conditions:
\begin{enumerate}
\item[(8)  ] $|[R,S]|=|[R,\widetilde{R}]||[R,\overline{R}]|$.

\item[(9)  ]  $\|[R,S]|=|[R,\overline{R}]||[\overline{R},S]|$.

\item[(10)] $|[R,\widetilde{R}]|=|[\overline{R},S]|$.

\item[(11)] $|[R,\overline{R}]|=|[\widetilde{R},S]|$.
\end{enumerate}

\end{theorem}

\begin{proof} (1) $\Rightarrow$ (2) by \cite[Lemma 3.7]{P2}. 

(2) $\Rightarrow$ (1). If the statement (2) holds, there exists $T\in [R,S]$ such that $T\cap\overline{R}=R$ and $T\overline{R}=S$. Then, \cite[Proposition 3.6]{P2} gives that $\mathrm{Supp}(\overline{R}/R)\cap\mathrm{Supp}(S/\overline{R})=\emptyset$.

(1) $\Rightarrow$ (3) by \cite[Proposition 3.6]{P2}.

(3) $\Rightarrow$ (4), (5), (6) and (7): Use Corollary~\ref{7.3} to get (4), Corollary~\ref{7.4} to get (5) and (6), and  Corollary~\ref{7.5} to get (7). Moreover, (3) and Corollary~\ref{7.3} give $\mathrm{Supp}(S/\widetilde{R})=\mathrm{Supp}(\overline{R}/R)$. 

(4) $\Rightarrow$ (1) by Proposition~\ref{4.13}(2).

(5), (6) or (7) $\Rightarrow$ (3) because, in each case, we have $S=\overline{R}\widetilde{R}$.

Assume now that $R\subset S$ has FIP. 

Then, obviously, (7) $\Rightarrow$ (8), (2) $\Rightarrow$ (9), (5) $\Rightarrow$ (10) and (6) $\Rightarrow$ (11). 

(9) $\Rightarrow$ (3) by Corollary~\ref{7.5}, which gives $|[R,\widetilde{R}]||[R,\overline{R}]|=|[R,\overline {\widetilde{R}}] |$, so that $|[R,S] |=|[R,\overline{\widetilde{R}}] |$, and then $S=\overline{\widetilde{R}}$. 

(8) $\Rightarrow$ (1): Using the map $\varphi$ of Lemma~\ref{7.1}, we get that $\{(T',T'')\in[R,\overline R]\times[\overline R,S]\mid\mathrm{Supp}_{T'}(\overline R/T')\cap\mathrm{Supp}_{T'}(T''/\overline R)=\emptyset\}=[R,\overline{R}]\times[\overline{R},S]$, so that $\mathrm{Supp}_{R}(\overline R/R)\cap\mathrm{Supp}_{R}(S/\overline R)=\emptyset$. 

(10) $\Rightarrow$ (3) and (11) $\Rightarrow$ (3) by Corollary~\ref{7.4}.
\end{proof}

\begin{example} We give an example where the results of Theorem~\ref{7.6} do not hold if $R\subseteq S$ has not  FCP. Set $R:=\mathbb{Z}_P$ and $S:=\mathbb{Q}[X]/(X^2)$, where $P\in \mathrm{Max}(\mathbb{Z})$. Then, $\widetilde R=\mathbb{Q}$ because $R\subset \widetilde R$ is Pr\"ufer (minimal) and $\widetilde R\subset S$ is  integral minimal. Set $M:=PR_P\in\mathrm {Max}(R)$ with $(R,M)$ a local ring. It follows that $M\in\mathrm{Supp}(\overline{R}/R)\cap\mathrm{Supp}(S/\overline{R})$ because $R\subset S$ is neither integral, nor Pr\"ufer. Similarly, $M\in\mathrm{Supp}(\overline{R}/R)\cap\mathrm{Supp}(\widetilde R/R)$. Indeed, $R\subset \overline R$ has not FCP. 
\end{example}

We end the paper by some length computations in the FCP case.

\begin{proposition}\label{7.8} Let $R\subseteq S$ be an FCP extension. The following statements hold:
\begin{enumerate}
\item $\ell [R,\widetilde R]=  \ell[\overline R, \vec R]$ and  $\ell [R,\overline R] = \ell [\widetilde R,\vec R]$

\item $\ell [R,\vec R]= \ell [R,\widetilde R] + \ell [\widetilde R,\vec R] = \ell[R,\overline R] + \ell [\overline R, \vec R]$

\item  $\ell [\overline R, \vec R]= |\mathrm{Supp}_{\overline R}(\vec R/ \overline R)| = \ell [R, \widetilde R] = |\mathrm{Supp}_R (\widetilde R/R)|$.
\end{enumerate}
\end{proposition}

\begin{proof} To prove (1), use the maps $\varphi_1$ and $\psi_1$ of Corollary~\ref{7.4}. Then (2) follows from \cite[Theorem 4.11]{DPP3} and (3) from \cite[Proposition 6.12]{DPP2}.
\end{proof}

\end{document}